\RequirePackage{fix-cm}

\documentclass[smallextended]{svjour3}
\smartqed
\usepackage{graphicx}

\usepackage{amsfonts}
\usepackage{amsmath}
\usepackage{csquotes}
\usepackage{algorithmicx}
\usepackage{algorithm}
\usepackage{algpseudocode}
\usepackage{color}

\usepackage{empheq}
\usepackage{physics}
\usepackage{comment}

\begin{document}

\title{Exactly Diagonal Gram Matrices in Jacobi Weighted Histopolation}

\titlerunning{Jacobi Weighted Histopolation meets Structured Linear Algebra}

\author{Allal Guessab\textsuperscript{*}
\thanks{* Retired from Universit\'e de Pau et des Pays
de l'Adour on 1 September 2025; currently an independent researcher.}
\and
Federico Nudo
\and
Stefano Serra-Capizzano
}


\institute{Allal Guessab \at
       Laboratoire de Mathématiques et de leurs Applications, UMR CNRS 5142, Université de Pau et des Pays de l'Adour (UPPA), 64000 Pau, France\\
             \email{allal.guessab@univ-pau.fr}
         \and
            Federico Nudo (corresponding author) \at
              Department of Mathematics and Computer Science, University of Calabria, Rende (CS), Italy\\
              \email{federico.nudo@unical.it}
 \and
            Stefano
Serra-Capizzano  \at
              University of Insubria, Department of Science and High Technology, via Valleggio 11, 22100  Como (Italy)\\
              University of Uppsala, Department of Information Technology\\ hus 10, L\"{a}gerhyddsv\"{a}gen 1, 75105 Uppsala (Sweden)\\
\email{s.serracapizzano@uninsubria.it}
}

\date{Received: date / Accepted: date}

\maketitle

\begin{abstract}
In the current work, we study univariate polynomial weighted histopolation on
$[-1,1]$, where the data are weighted integrals over a family of intervals.
After choosing a polynomial basis, the weighted moment conditions lead to a
histopolation matrix whose structure depends on the weight and on the geometry
of the cells. We investigate its nonsingularity, which guarantees unisolvence,
together with exact diagonality of its Gram matrix, which allows its singular
values and spectral condition number to be determined explicitly. For families of intervals whose endpoints belong to a fixed grid, we characterize unisolvence in terms of the connectedness of the associated
endpoint graph. In the unisolvent case, this graph is a tree, and the unique paths joining consecutive grid points provide an explicit expression for the inverse matrix, while their lengths determine its infinity norm. For the nested and supplemented sliding-window families, we show that the corresponding interval-cell incidence matrices are related, up to diagonal sign matrices, by inversion. This identity gives explicit formulas for the singular values of both matrices, and shows that their condition numbers in the two-norm coincide and grow linearly with the matrix size. Moreover, it yields the limiting singular value distributions of the two matrix sequences.
We also establish a general diagonalization criterion based on discrete
weighted orthogonality. The criterion recovers the first kind Chebyshev
construction and leads to a diagonal configuration for the constant weight
based on discrete sine orthogonality. For interval families with a connected
endpoint graph, the corresponding moment vectors define an inner product on
the polynomial space and lead to a monic basis with a diagonal weighted Gram
matrix. Finally, we derive reduction formulas for cell moments associated
with generalized Jacobi weights and introduce an alternative basis for
shifted Jacobi weights. Applied to the Chebyshev weight of the fourth kind,
this basis, together with a correction of one nonconstant element, yields an
exactly diagonal Gram matrix.

\keywords{Orthogonal polynomials \and Gram matrix \and Histopolation \and Weighted approximation \and Jacobi polynomials}

\subclass{65D05}
\end{abstract}

\section{Introduction}
\label{sec:introduction}

In many reconstruction problems, including tomography, imaging, and numerical
simulation, the available information is not given by pointwise values of the
unknown function but by averages or integrals over intervals, cells, faces,
or more general geometric regions; see, e.g.,
\cite{Kak:2001:POC,Natterer:2001:TMO,Palamodov:2016:RFI}. In such situations, the reconstruction procedure should reflect the integral nature of the data. This is the basic idea behind \textit{histopolation},
where the degrees of freedom are average values or moment-type functionals~\cite{Bruni:2024:PIO}. Several histopolation methods and their
properties have been studied in recent years; see~\cite{Kirsiaed:2020:RSH,Bruni:2025:OTC,Bruni:2025:TFP,Bruni:2026:AAF}.
When the degrees of freedom are defined by weighted integral functionals,
the resulting problem is referred to as \textit{weighted histopolation}~\cite{Demichelis:1995:GAO,Guessab:2026:SDO}. Such constructions arise when
the data are averaged with respect to a nonuniform density represented by a
weight function~\cite{Bose:1965:FSA}. The presence of the weight modifies the
moment conditions and may affect fundamental properties of the reconstruction
problem, including unisolvence and stability~\cite{Guessab:2026:QWH}. After a
basis for the reconstruction space has been chosen, these conditions give
rise to a structured moment matrix whose properties depend on the weight, the
polynomial basis, and the geometry of the cells. A first question is whether the prescribed weighted integral data uniquely
determine a polynomial in the reconstruction space. This is the unisolvence
problem, which is equivalent to the nonsingularity of the corresponding
moment matrix. A characterization of unisolvence for univariate polynomial
weighted histopolation on families of intervals is given
in~\cite{Guessab:2026:WDH}. Nonsingularity, however, does not by itself
guarantee numerical stability. The smallest singular value may still be close
to zero, and the spectral condition number may become large~\cite{Bruni:2025:OTC}. This motivates the study of configurations for which
the structure of the moment matrix can be described explicitly.

In this paper, we consider univariate polynomial weighted histopolation on
$[-1,1]$ associated with the Jacobi weights
\[
\omega_{\alpha,\beta}(x)
=
(1-x)^\alpha(1+x)^\beta,
\quad
\alpha,\beta>-1.
\]
We first study families of intervals whose endpoints belong to a fixed grid.
With each family, we associate an undirected graph whose vertices are the grid points and
whose edges join the endpoints of the intervals. We prove that unisolvence is equivalent to the connectedness of the endpoint graph. In the unisolvent case, the considered graph is a tree, and the unique paths joining consecutive grid points lead to an explicit representation of the inverse of the interval matrix, while their lengths provide a direct formula for its infinity norm. For the nested and supplemented sliding-window families, we further show that one interval matrix is similar, through a diagonal sign matrix, to the inverse of the other one. Consequently, their condition
numbers in the two-norm coincide, are determined explicitly, and grow
linearly with the matrix size. We also determine the limiting singular value distributions of the corresponding matrix sequences. We then turn to cells of constant angular length. For the Chebyshev weight of the first kind, discrete cosine orthogonality implies that the Gram matrix of the normalized histopolation matrix is exactly diagonal. This gives explicit formulas for its diagonal entries, singular values, and spectral condition number. We also prove a converse result. Whenever the number of
cells is at least three, requiring this diagonal structure for every
admissible angular half-length characterizes the Jacobi parameters $\alpha=\beta=-\frac{1}{2}$. The Chebyshev construction is then placed within a more general framework
based on discrete weighted orthogonality. We establish a criterion ensuring
the exact diagonality of a weighted Gram matrix and recover the first kind
Chebyshev case as a particular instance. The same criterion yields a configuration for the constant weight based on discrete sine orthogonality. For interval families with a connected endpoint graph, the corresponding
moment vectors define an inner product on the polynomial space. The unique
monic basis orthogonal with respect to this inner product gives a diagonal
weighted Gram matrix. We also show that, when only one nonconstant column is
not orthogonal to the constant column, exact diagonality can be recovered by
subtracting a suitable multiple of the constant polynomial from the corresponding basis element. Finally, we consider generalized and shifted Jacobi weights. The associated cell moments are reduced to finite linear combinations of classical Jacobi
cell moments, whose positive-degree terms admit explicit representations in
terms of endpoint values. Subsequently, we introduce an alternative polynomial basis
for shifted Jacobi weights. For the Chebyshev weight of the fourth kind, this
basis, together with the correction described above, yields an exactly diagonal Gram matrix.

In addition to the technical results, the main value of the present work is a fruitful connection among approximation theory, structured matrices, and asymptotic/numerical linear algebra; see also~\cite{Guessab:2026:SDO} for other findings that use a similar interplay among the same fields.

The paper is organized as follows. In
Section~\ref{sec:combinatorialUnisolvence}, we study interval families generated from a fixed grid, characterize unisolvence through the connectedness of the associated endpoint graph, and derive an explicit formula for the
inverse of the interval matrix. In Section~\ref{sec3}, we consider cells of
constant angular length and establish the exact diagonality results for the
Chebyshev weight of the first kind, together with the formulas for the
singular values and spectral condition number. In Section~\ref{sec:diagonalizationCriterion}, we develop the general diagonalization criterion based on discrete weighted orthogonality, construct the associated monic orthogonal basis, and present further configurations based on discrete sine orthogonality. In Section~\ref{subsec:generalShiftedJacobiDifferenceQuotient}, we derive the
moment reduction formulas for generalized and shifted Jacobi weights and
apply the resulting constructions to the Chebyshev weight of the fourth kind.

\section{A combinatorial unisolvence criterion}
\label{sec:combinatorialUnisolvence}

In this section, we consider weighted histopolation on interval families
obtained from a fixed grid. For these families, the unisolvence problem has
an algebraic characterization in terms of an interval matrix, as described
in~\cite{Guessab:2026:WDH}. We use this characterization as a starting point
and then give an equivalent geometric form in terms of a graph on the grid
points. The considered geometric point of view also gives an explicit formula for the
inverse of the interval matrix.

Let
\begin{equation}
\label{gridXN}
X_N=\left\{x_0,\ldots,x_N\right\},
\quad
-1=x_0<x_1<\cdots<x_N=1,
\end{equation}
and set
\[
e_j=\left[x_{j-1},x_j\right],
\quad j=1,\ldots,N.
\]
Assume that $\omega\in L^1(-1,1)$ satisfies
\[
\omega(x)>0
\quad\text{for almost every }x\in(-1,1).
\]
We define the weighted elementary moment matrix
$M_{\mathcal E}\in\mathbb R^{N\times N}$ by
\[
\left[M_{\mathcal E}\right]_{j,k+1}
=
\int_{e_j}P_k(x)\omega(x)dx,
\quad
j=1,\ldots,N,
\quad
k=0,\ldots,N-1,
\]
where $P_0,\ldots,P_{N-1}$ form a basis of $\Pi_{N-1}$ and satisfy
\[
\operatorname{deg}\left(P_k\right)=k,
\quad
k=0,\ldots,N-1.
\]
Let
\[
\mathcal S=\left\{s_1,\ldots,s_N\right\}
\]
be a family of $N$ pairwise distinct intervals whose endpoints belong to the
grid~\eqref{gridXN}.  Thus, for any $i=1,\ldots,N$, there exist integers
$\ell_i,r_i$ such that
\[
0\leq \ell_i<r_i\leq N,
\quad
s_i=\left[x_{\ell_i},x_{r_i}\right].
\]
We associate with $\mathcal S$ the interval matrix
$A_{\mathcal S}\in\mathbb R^{N\times N}$ defined by
\[
\left[A_{\mathcal S}\right]_{ij}
=
\begin{cases}
1, & \ell_i<j\leq r_i,\\
0, & \text{otherwise}.
\end{cases}
\]
Equivalently,
\[
\left[A_{\mathcal S}\right]_{ij}=1
\]
if and only if the elementary
cell $e_j$ is contained in $s_i$. Let $M_{\mathcal S}\in\mathbb R^{N\times N}$ be the corresponding moment
matrix, namely
\[
\left[M_{\mathcal S}\right]_{i,k+1}
=
\int_{s_i}P_k(x)\omega(x)dx,
\quad
i=1,\ldots,N,
\quad
k=0,\ldots,N-1.
\]

\begin{proposition}
\label{prop:subgridIntervalFactorization}
The matrix $M_{\mathcal S}$ satisfies
\begin{equation}
\label{eq:subgridIntervalFactorization}
M_{\mathcal S}=A_{\mathcal S}M_{\mathcal E}.
\end{equation}
Moreover,
\begin{equation}
\label{eq:subgridOrdinaryGram}
M_{\mathcal S}^{\top}M_{\mathcal S}
=
M_{\mathcal E}^{\top}
A_{\mathcal S}^{\top}A_{\mathcal S}
M_{\mathcal E}.
\end{equation}
\end{proposition}

\begin{proof}
For each $i$, we have
\[
s_i=\left[x_{\ell_i},x_{r_i}\right]=\bigcup_{j={\ell_i}+1}^{r_i}e_j.
\]
Hence, by additivity of the integral, we have
\[
\left[M_{\mathcal S}\right]_{i,k+1}
=
\sum_{j=\ell_i+1}^{r_i}
\int_{e_j}P_k(x)\omega(x)dx
=
\sum_{j=1}^{N}
\left[A_{\mathcal S}\right]_{ij}
\left[M_{\mathcal E}\right]_{j,k+1},
\]
which proves~\eqref{eq:subgridIntervalFactorization}. Therefore,
\[
M_{\mathcal S}^{\top}M_{\mathcal S}
=
\left(A_{\mathcal S}M_{\mathcal E}\right)^{\top}
\left(A_{\mathcal S}M_{\mathcal E}\right)
=
M_{\mathcal E}^{\top}
A_{\mathcal S}^{\top}A_{\mathcal S}
M_{\mathcal E},
\]
which gives~\eqref{eq:subgridOrdinaryGram}.
\end{proof}

\begin{remark}
If $A_{\mathcal S}$ is nonsingular, then
\[
W_{\mathcal S}=A_{\mathcal S}^{\top}A_{\mathcal S}
\]
is symmetric positive definite. Therefore, diagonalization of
$M_{\mathcal S}^{\top}M_{\mathcal S}$ is equivalent to diagonalization of
the weighted Gram matrix
\[
M_{\mathcal E}^{\top}W_{\mathcal S}M_{\mathcal E}.
\]
\end{remark}
Under the assumptions above, the elementary family
\[
\mathcal E=\left\{e_1,\ldots,e_N\right\}
\]
is unisolvent on $\Pi_{N-1}$. Hence $M_{\mathcal E}$ is nonsingular~\cite{Bruni:2024:PIO,Bruni:2025:OTC}.
Moreover, the algebraic characterization described in~\cite{Guessab:2026:WDH}
states that
\[
M_{\mathcal S}
\quad\text{is nonsingular}
\quad\Longleftrightarrow\quad
A_{\mathcal S}
\quad\text{is nonsingular}.
\]
Equivalently, the weighted histopolation problem on $\mathcal S$ is
unisolvent if and only if the interval matrix $A_{\mathcal S}$ is
nonsingular.

We now give a geometric interpretation of the same condition.

\subsection{A geometric characterization result}
\label{subsec:endpointTrees}

Let
\[
\mathcal S=\left\{s_1,\ldots,s_N\right\},
\quad
s_i=\left[x_{\ell_i},x_{r_i}\right],
\quad i=1,\ldots,N.
\]
We now associate a graph with the endpoints of the intervals. The vertices of the graph are the indices of the grid points, i.e.
\[
\mathcal V_N=\{0,1,\ldots,N\}.
\]
Each interval $s_i$ gives one edge joining its two endpoint indices
$\ell_i$ and $r_i$. Thus we define the undirected graph
\[
\mathcal G_{\mathcal S}
=
(\mathcal V_N,\mathcal E_{\mathcal S}),
\]
where
\[
\mathcal E_{\mathcal S}
=
\left\{
\left\{\ell_i,r_i\right\}\,:\, i=1,\ldots,N
\right\}.
\]
In this way, the interval $s_i=\left[x_{\ell_i},x_{r_i}\right]$ is represented by
the edge $\left\{\ell_i,r_i\right\}$.

We recall the elementary  terminology used below. A path in
$\mathcal G_{\mathcal S}$ is a finite sequence of pairwise distinct vertices
\[
v_0,v_1,\ldots,v_m
\]
such that
\[
\left\{v_{k-1},v_k\right\}\in\mathcal E_{\mathcal S},
\quad k=1,\ldots,m.
\]
The graph $\mathcal G_{\mathcal S}$ is connected if and only if every two vertices in
$\mathcal V_N$ can be joined by a path.

\begin{theorem}
\label{thm:endpointTreeCriterion}
The following statements are equivalent:
\begin{enumerate}
\item $A_{\mathcal S}$ is nonsingular;
\item $\mathcal G_{\mathcal S}$ is connected.
\end{enumerate}
Consequently, the unisolvence condition
\[
\det\left(A_{\mathcal S}\right)\neq0
\]
is equivalent to the fact that the intervals of $\mathcal S$ connect all
grid points.
\end{theorem}

\begin{proof}
Assume first that $\mathcal G_{\mathcal S}$ is connected. We prove that
$A_{\mathcal S}$ is nonsingular. Let
\[
\boldsymbol c=\left(c_1,\ldots,c_N\right)^{\top}\in\ker\left(A_{\mathcal S}\right).
\]
We define the cumulative sums
\begin{equation*}
u_0=0,
\quad
u_j=\sum_{k=1}^{j}c_k,
\quad j=1,\ldots,N.
\end{equation*}
Then
\begin{equation}\label{csa}
    c_j=u_j-u_{j-1},
\quad j=1,\ldots,N.
\end{equation}
For an interval
\[
s_i=\left[x_{\ell_i},x_{r_i}\right]\in\mathcal{S},
\]
the definition of $A_{\mathcal S}$ gives
\begin{equation*}
0=\left[A_{\mathcal S}\boldsymbol c\right]_i
=
\sum_{j=\ell_i+1}^{r_i}c_j
=
u_{r_i}-u_{\ell_i}.
\end{equation*}
Hence
\[
u_{r_i}=u_{\ell_i},
\quad i=1,\ldots,N.
\]
Let $j\in\{1,\ldots,N\}$. Since $\mathcal G_{\mathcal S}$ is connected,
there exists a path
\[
v_0,v_1,\ldots,v_m, \quad \text{such that} \quad
\left\{v_{k-1},v_k\right\}\in\mathcal E_{\mathcal S},
\quad k=1,\ldots,m,
\]
with
\[
v_0=0,
\quad
v_m=j.
\]
For each edge $\left\{v_{k-1},v_k\right\}$ of the path, there exists an interval
$s_i=\left[x_{\ell_i},x_{r_i}\right]\in\mathcal S$ such that
\[
\ell_i=\min\left\{v_{k-1},v_k\right\}, \quad r_i=\max\left\{v_{k-1},v_k\right\}.
\]
Therefore
\[
u_{v_0}=u_{v_1}=\cdots=u_{v_m}.
\]
Since $u_{v_0}=u_0=0$, we obtain
\[
u_j=u_{v_m}=0.
\]
Since this holds for every $j=1,\ldots,N$, we have
\[
u_0=u_1=\cdots=u_N=0.
\]
Using~\eqref{csa}, it follows that
\[
c_j=u_j-u_{j-1}=0,
\quad j=1,\ldots,N.
\]
Hence $\boldsymbol c=\boldsymbol 0$. Therefore
\[
\ker\left(A_{\mathcal S}\right)=\{\boldsymbol 0\}.
\]
Conversely, assume that $\mathcal G_{\mathcal S}$ is not connected. Then
there exists a vertex
\[
q\in\{1,\ldots,N\}
\]
which cannot be joined to $0$ by any path. Let $\mathcal C$ be the set of
vertices which can be joined to $q$ by a path. Then
\[
q\in\mathcal C,
\quad
0\notin\mathcal C.
\]
Define
\[
\boldsymbol{u}=\left(u_0,\dots,u_N\right)^{\top}\in\mathbb{R}^{N+1}
\]
such that
\[
u_j=
\begin{cases}
1, & j\in\mathcal C,\\
0, & j\notin\mathcal C.
\end{cases}
\]
Then
\[
u_0=0,
\quad
u_q=1.
\]
Hence, for any edge
$\left\{\ell_i,r_i\right\}$, we have
\begin{equation}\label{csanew6}
    u_{r_i}=u_{\ell_i}.
\end{equation}
Now set
\[
\boldsymbol c=\left(c_1,\ldots,c_N\right)^{\top}\in\mathbb{R}^N
\]
such that
\[
c_j=u_j-u_{j-1},
\quad j=1,\ldots,N.
\]
Since $u_0=0$ and $u_q=1$, the vector
\[
\boldsymbol c=\left(c_1,\ldots,c_N\right)^{\top}
\]
is nonzero. On the other hand, by~\eqref{csanew6}, we have
\[
\left[A_{\mathcal S}\boldsymbol c\right]_i
=
u_{r_i}-u_{\ell_i}
=
0,
\quad i=1,\ldots,N.
\]
Thus $A_{\mathcal S}$ has a nontrivial kernel and is singular.
\end{proof}

\begin{example}
\label{ex:endpointGraphConfigurations}
The connectedness criterion in
Theorem~\ref{thm:endpointTreeCriterion} can often be verified directly,
without computing the determinant of the interval matrix. We illustrate this
with three families defined for arbitrary values of $N\ge 3$.

\begin{itemize}

\item \textit{A zig-zag family.}\\
We order the vertices of $\mathcal V_N$ as follows:
\[
v_i=
\begin{cases}
\dfrac{i}{2}, & \text{if $i$ is even},\\[2mm]
N-\dfrac{i-1}{2}, & \text{if $i$ is odd},
\end{cases}
\quad
i=0,\ldots,N.
\]
The vertices with even indices in this sequence are
\[
0,1,\ldots,\left\lfloor\frac{N}{2}\right\rfloor,
\]
whereas those with odd indices are
\[
N,N-1,\ldots,\left\lfloor\frac{N}{2}\right\rfloor+1.
\]
For
$i=1,\ldots,N$, set
\[
\ell_i=\min\left\{v_{i-1},v_i\right\},
\quad
r_i=\max\left\{v_{i-1},v_i\right\},
\]
and define
\[
s_i=\left[x_{\ell_i},x_{r_i}\right].
\]
Now we consider the interval family
\[
\mathcal S_{\mathrm{zig}}
=
\left\{s_1,\ldots,s_N\right\}.
\]
By construction, the edge associated with $s_i$ is
\[
\left\{v_{i-1},v_i\right\},
\quad
i=1,\ldots,N.
\]
Therefore, the sequence
\[
v_0,v_1,\ldots,v_N
\]
defines a simple path in
$\mathcal G_{\mathcal S_{\mathrm{zig}}}$ containing all the vertices of
$\mathcal V_N$, as illustrated in Fig.~\ref{fig:zigzag}. Hence,
$\mathcal G_{\mathcal S_{\mathrm{zig}}}$ is connected and, by
Theorem~\ref{thm:endpointTreeCriterion},
$\mathcal S_{\mathrm{zig}}$ is unisolvent.

\item \textit{An interior star family.}\\
Fix a vertex
\[
q\in\{1,\ldots,N-1\}
\]
and consider
\[
\mathcal S_{\mathrm{star}}^{(q)}
=
\left\{
\left[x_i,x_q\right]
\,:\,
i=0,\ldots,q-1
\right\}
\cup
\left\{
\left[x_q,x_i\right]
\,:\,
i=q+1,\ldots,N
\right\}.
\]
The associated endpoint graph has edge set
\[
\mathcal E_{\mathcal S_{\mathrm{star}}^{(q)}}
=
\left\{
\left\{q,i\right\}
\,:\,
i\in\mathcal V_N\setminus\{q\}
\right\}.
\]
It is therefore a star centered at the interior vertex $q$, as illustrated in
Fig.~\ref{fig:intstar}. Since any vertex is connected to $q$, the graph is
connected. It follows from Theorem~\ref{thm:endpointTreeCriterion} that
$\mathcal S_{\mathrm{star}}^{(q)}$ is unisolvent.

\item \textit{A disconnected family.}\\
Define
\[
\mathcal S_{\mathrm{disc}}
=
\left\{
\left[x_0,x_i\right]
\,:\,
i=1,\ldots,N-1
\right\}
\cup
\left\{
\left[x_1,x_{N-1}\right]
\right\}.
\]
The corresponding endpoint graph has edges
\[
\{0,1\},\{0,2\},\ldots,\{0,N-1\},
\{1,N-1\}.
\]
All these edges have their endpoints in
\[
\{0,\ldots,N-1\},
\]
whereas the vertex $N$ is isolated, as illustrated in
Fig.~\ref{fig:discon}. Therefore,
$\mathcal G_{\mathcal S_{\mathrm{disc}}}$ is not connected and
Theorem~\ref{thm:endpointTreeCriterion} shows that
$\mathcal S_{\mathrm{disc}}$ is not unisolvent.
\end{itemize}
\end{example}

\begin{figure}[ht]
  \centering
  \includegraphics[width=0.49\textwidth]{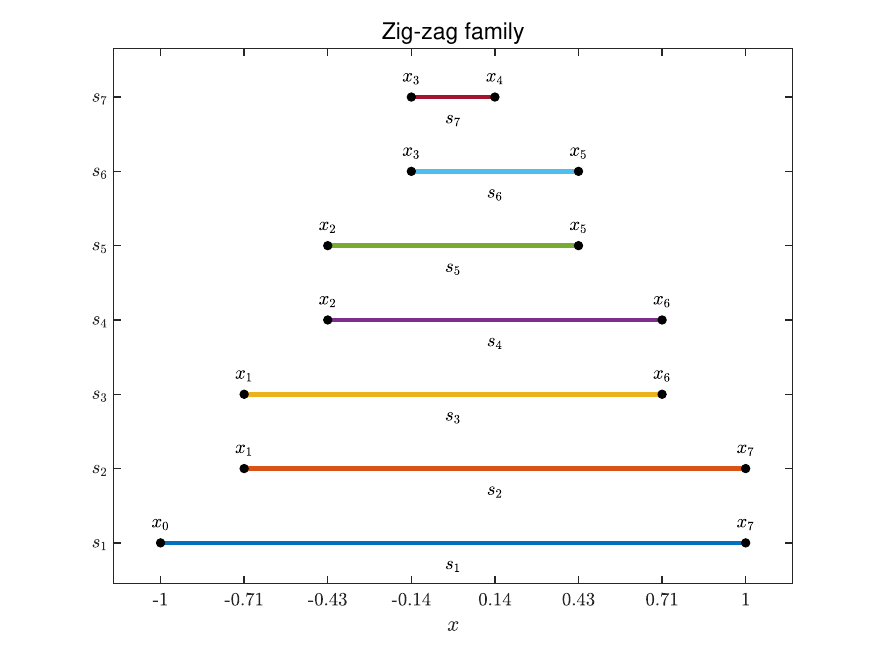}
  \hfill
  \includegraphics[width=0.49\textwidth]{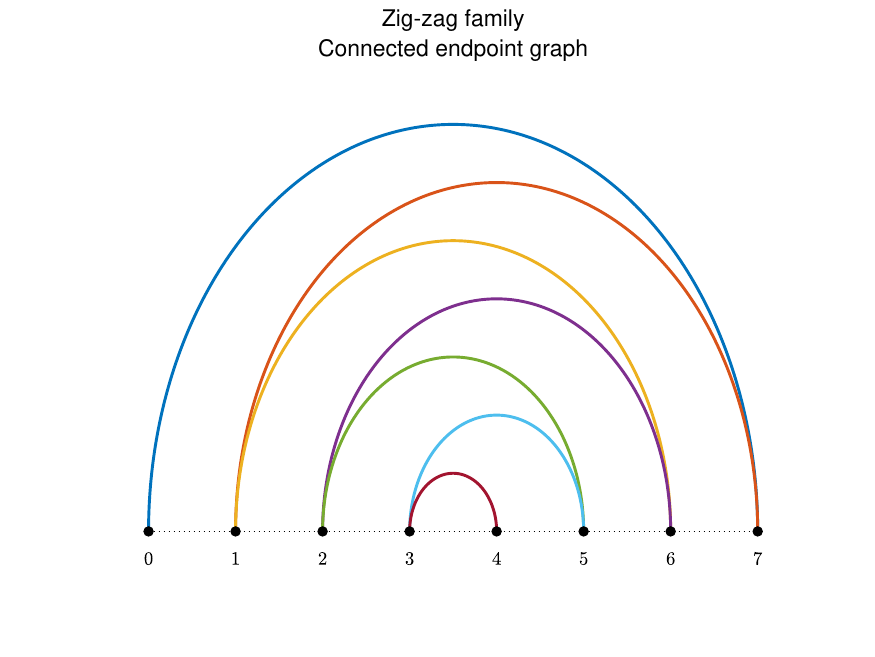}
  \caption{The zig-zag configuration for $N=7$. Left: the interval family
  $\mathcal S_{\mathrm{zig}}$. Right: the associated endpoint graph
  $\mathcal G_{\mathcal S_{\mathrm{zig}}}$.}
  \label{fig:zigzag}
\end{figure}

\begin{figure}[ht]
  \centering
  \includegraphics[width=0.49\textwidth]{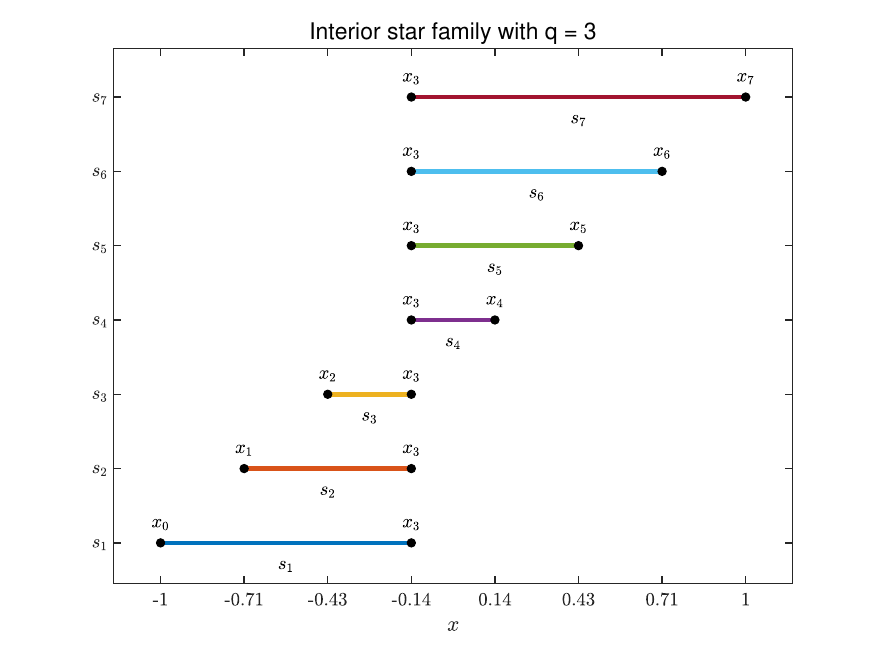}
  \hfill
  \includegraphics[width=0.49\textwidth]{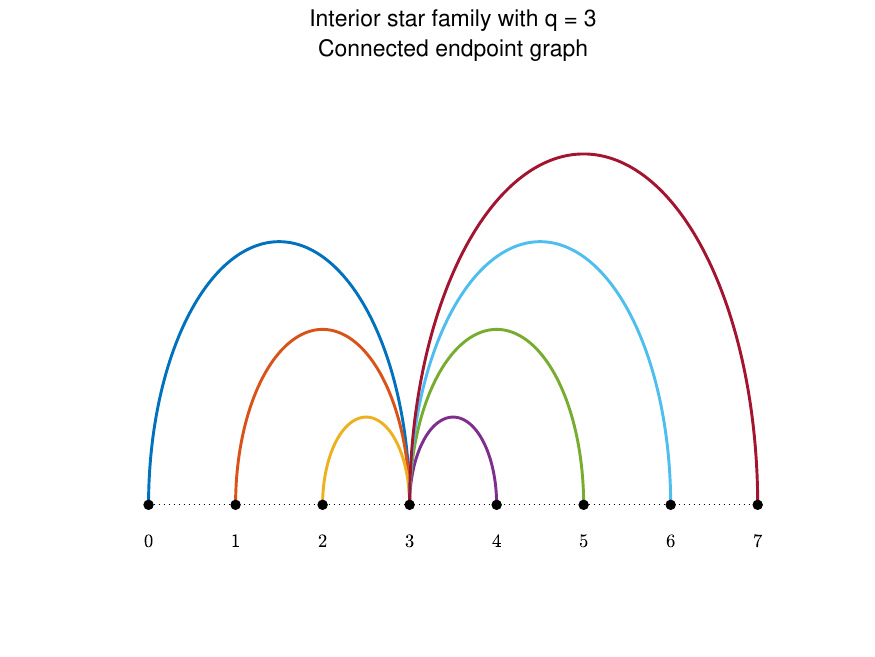}
  \caption{The interior star configuration for $N=7$. Left: the interval
  family $\mathcal S_{\mathrm{star}}^{(q)}$. Right: the associated endpoint
  graph $\mathcal G_{\mathcal S_{\mathrm{star}}^{(q)}}$.}
  \label{fig:intstar}
\end{figure}

\begin{figure}[ht]
  \centering
  \includegraphics[width=0.49\textwidth]{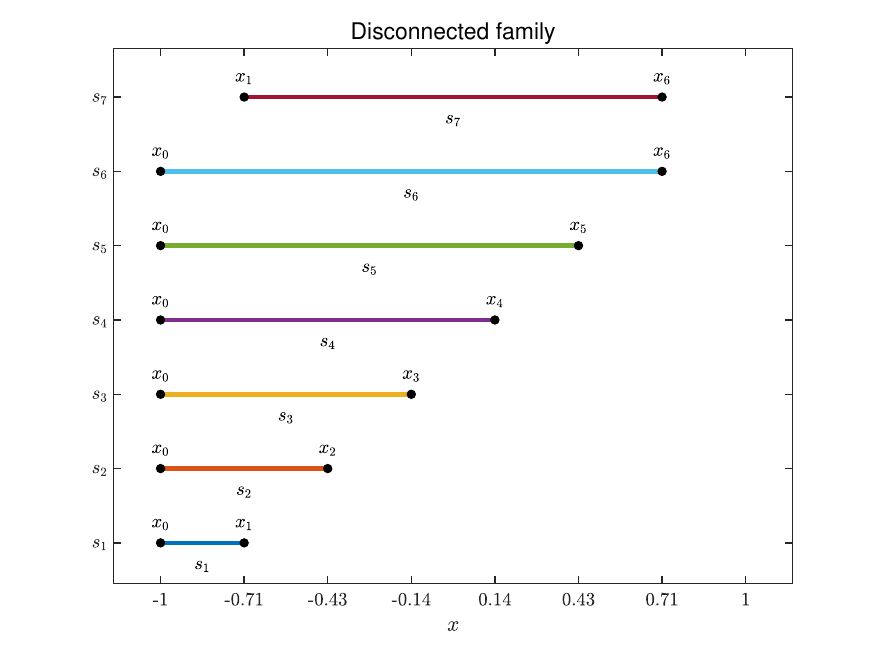}
  \hfill
  \includegraphics[width=0.49\textwidth]{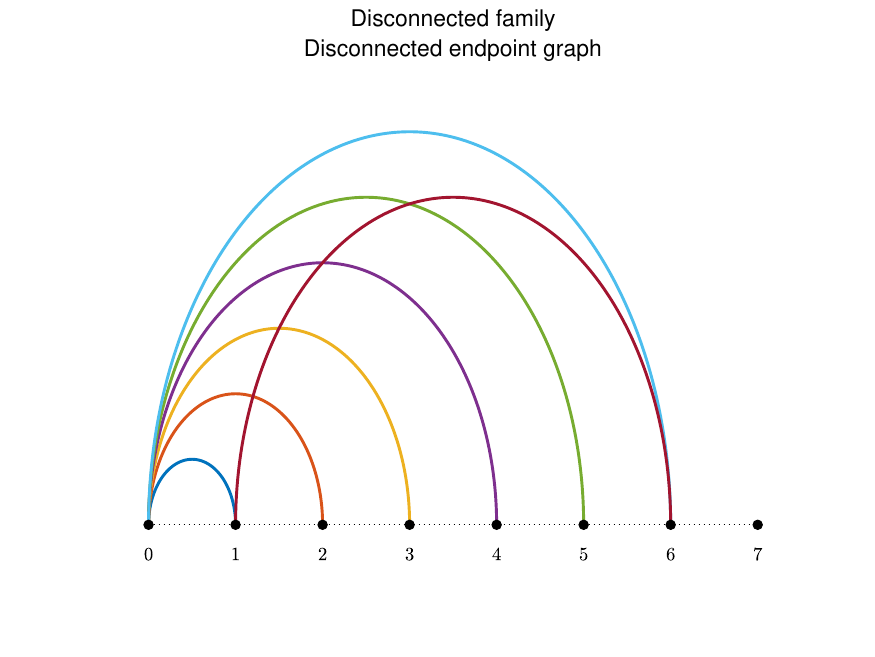}
  \caption{The disconnected configuration for $N=7$. Left: the interval
  family $\mathcal S_{\mathrm{disc}}$. Right: the associated endpoint graph
  $\mathcal G_{\mathcal S_{\mathrm{disc}}}$.}
  \label{fig:discon}
\end{figure}

\begin{remark}
Assume that $A_{\mathcal S}$ is nonsingular. Then, by
Theorem~\ref{thm:endpointTreeCriterion}, the graph
$\mathcal G_{\mathcal S}$ is connected. Moreover, it has $N+1$ vertices and
$N$ edges. We first show that $\mathcal G_{\mathcal S}$ has no cycles. Suppose, by
contradiction, that it contains a cycle
\[
v_0,v_1,\ldots,v_m, \quad m\geq 3,
\quad
v_0=v_m,
\]
where the vertices $v_0,\ldots,v_{m-1}$ are distinct. Remove from
$\mathcal G_{\mathcal S}$ the edge
\[
\left\{v_{m-1},v_m\right\}.
\]
The graph remains connected.
Therefore, after removing one edge of the cycle, we would obtain a
connected graph with $N+1$ vertices and $N-1$ edges. The latter is impossible,
because any connected graph with $N+1$ vertices has at least $N$ edges.
Hence $\mathcal G_{\mathcal S}$ has no cycles.

Since $\mathcal G_{\mathcal S}$ is connected and has no cycles, between
any two vertices there is a unique path. Indeed, if two distinct paths
joined the same pair of vertices, their union would contain a cycle.
\end{remark}

We now assign an orientation to each edge. The edge associated with
\[
s_i=\left[x_{\ell_i},x_{r_i}\right]
\]
is oriented from $\ell_i$ to $r_i$, and the corresponding oriented
edge is denoted by $\left(\ell_i,r_i\right)$. This orientation is used only to
assign signs to the edges appearing in a path. For each elementary interval
\[
e_j=\left[x_{j-1},x_j\right],
\]
let $\mathcal P_j$ be the unique path in $\mathcal G_{\mathcal S}$
joining $j-1$ to $j$. We define
\begin{equation*} \varepsilon_{ji} = \begin{cases} 1, & \text{if the path } \mathcal P_j \text{ uses the edge of } s_i \text{ from } \ell_i \text{ to } r_i,\\ -1, & \text{if the path } \mathcal P_j \text{ uses the edge of } s_i \text{ from } r_i \text{ to } \ell_i,\\ 0, & \text{if the path } \mathcal P_j \text{ does not use the edge of } s_i. \end{cases} \end{equation*}
Let
\begin{equation}\label{eq_BS}
    B_{\mathcal S}
=
\left[\varepsilon_{ji}\right]_{j,i=1}^{N}\in\mathbb{R}^{N \times N}.
\end{equation}
The preceding path construction leads to an explicit representation of the inverse of the interval matrix.
\begin{theorem}
\label{thm:pathInverseIntervalMatrix}
If $A_{\mathcal S}$ is nonsingular, then
\[
A_{\mathcal S}^{-1}
=
B_{\mathcal S}.
\]
\end{theorem}

\begin{proof}
Let
\[
\boldsymbol c=\left(c_1,\ldots,c_N\right)^{\top}\in\mathbb R^N
\]
and set
\[
\boldsymbol y=A_{\mathcal S}\boldsymbol c.
\]
As before, define
\[
u_0=0,
\quad
u_j=\sum_{k=1}^{j}c_k,
\quad j=1,\ldots,N.
\]
Then
\[
c_j=u_j-u_{j-1},
\quad j=1,\ldots,N.
\]
Moreover, for any interval
\[
s_i=\left[x_{\ell_i},x_{r_i}\right],
\]
we have
\[
y_i
=
\left[A_{\mathcal S}\boldsymbol c\right]_i
=
\sum_{k=\ell_i+1}^{r_i}c_k
=
u_{r_i}-u_{\ell_i}.
\]
Fix $j\in\{1,\ldots,N\}$. Write the unique path $\mathcal P_j$ joining
$j-1$ to $j$ as
\[
\mathcal P_j=(v_0,\ldots,v_m),
\quad
v_0=j-1,
\quad
v_m=j.
\]
For each $k=1,\ldots,m$, let $i_k$ be the index of the interval whose
edge is
\[
\left\{v_{k-1},v_k\right\}=\left\{\ell_{i_k},r_{i_k}\right\}.
\]
Then either
\[
\left(v_{k-1},v_k\right)=\left(\ell_{i_k},r_{i_k}\right)
\]
or
\[
\left(v_{k-1},v_k\right)=\left(r_{i_k},\ell_{i_k}\right).
\]
In the first case, by the definition of $\varepsilon_{j i_k}$, we have
\[
\varepsilon_{j i_k}=1,
\]
and therefore
\[
\varepsilon_{j i_k}y_{i_k}
=
y_{i_k}
=
u_{r_{i_k}}-u_{\ell_{i_k}}
=
u_{v_k}-u_{v_{k-1}}.
\]
In the second case,
\[
\varepsilon_{j i_k}=-1,
\]
and therefore
\[
\varepsilon_{j i_k}y_{i_k}
=
-y_{i_k}
=
u_{\ell_{i_k}}-u_{r_{i_k}}
=
u_{v_k}-u_{v_{k-1}}.
\]
Thus, in both cases, we have
\[
\varepsilon_{j i_k}y_{i_k}
=
u_{v_k}-u_{v_{k-1}},
\quad
k=1,\ldots,m.
\]
Since $\varepsilon_{ji}=0$ for the edges which do not belong to
$\mathcal P_j$, it follows that
\[
\sum_{i=1}^{N}\varepsilon_{ji}y_i
=
\sum_{k=1}^{m}\varepsilon_{j i_k}y_{i_k}
=
\sum_{k=1}^{m}
\left(u_{v_k}-u_{v_{k-1}}\right)
=
u_{v_m}-u_{v_0}.
\]
Using $v_0=j-1$ and $v_m=j$, we get
\[
\sum_{i=1}^{N}\varepsilon_{ji}y_i
=
u_j-u_{j-1}
=
c_j.
\]
Since this holds for every $j=1,\ldots,N$, we have
\[
B_{\mathcal S}\boldsymbol y
=
\boldsymbol c.
\]
Since $\boldsymbol y=A_{\mathcal S}\boldsymbol c$, we obtain
\[
B_{\mathcal S}A_{\mathcal S}\boldsymbol c
=
\boldsymbol c
\]
for every $\boldsymbol c\in\mathbb R^N$. Hence
\[
B_{\mathcal S}A_{\mathcal S}=I_N.
\]
Since $A_{\mathcal S}$ is nonsingular, the claim follows.
\end{proof}
\begin{corollary}
\label{cor:geometricNormsInverse}
If $A_{\mathcal S}$ is nonsingular, then
\[
\left\|A_{\mathcal S}^{-1}\right\|_{\infty}
=
\max_{1\leq j\leq N}
\left|\mathcal P_j\right|,
\]
where $\left|\mathcal P_j\right|$ denotes the number of edges in the path joining
$j-1$ to $j$.
\end{corollary}

\begin{proof}
By Theorem~\ref{thm:pathInverseIntervalMatrix}, the $j$-th row of
$A_{\mathcal S}^{-1}$ has one nonzero entry for each edge of
$\mathcal P_j$, and each such entry is equal to either $1$ or $-1$.
Therefore the sum of the absolute values in the $j$-th row is exactly
$\left|\mathcal P_j\right|$. Taking the maximum over $j$ gives the formula for
$\left\|A_{\mathcal S}^{-1}\right\|_{\infty}$.
\end{proof}

\begin{corollary}
\label{cor:infinityConditionNumberFactorization}
Assume that $A_{\mathcal S}$ is nonsingular, and set
\[
W_{\max}(\mathcal S)
=
\max_{1\leq i\leq N}
\left(r_i-\ell_i\right),
\quad
L_{\max}(\mathcal S)
=
\max_{1\leq j\leq N}
\left|\mathcal P_j\right|.
\]
Then
\[
\kappa_{\infty}
\left(
A_{\mathcal S}
\right)
=
W_{\max}(\mathcal S)
L_{\max}(\mathcal S).
\]
\end{corollary}
\begin{proof}
The $i$-th row of $A_{\mathcal S}$ contains exactly
$r_i-\ell_i$ nonzero entries, all equal to one. Therefore,
\[
\left\|A_{\mathcal S}\right\|_{\infty}
=
\max_{1\leq i\leq N}
\left(r_i-\ell_i\right)
=
W_{\max}(\mathcal S).
\]
On the other hand, Corollary~\ref{cor:geometricNormsInverse} gives
\[
\left\|A_{\mathcal S}^{-1}\right\|_{\infty}
=
\max_{1\leq j\leq N}
\left|\mathcal P_j\right|
=
L_{\max}(\mathcal S).
\]
The result now follows from the definition of the infinity condition
number.
\end{proof}

\begin{remark}
The preceding identity separates two distinct features of the interval
family. The factor $W_{\max}(\mathcal S)$ measures the largest number of
elementary cells contained in a member of $\mathcal S$, whereas
$L_{\max}(\mathcal S)$ is the largest graph distance, in the endpoint
tree, between two consecutive grid vertices.
\end{remark}

\begin{remark}
Assume that $A_{\mathcal S}$ is nonsingular. Then, combining
$M_{\mathcal S}=A_{\mathcal S}M_{\mathcal E}$
with Corollary~\ref{cor:infinityConditionNumberFactorization}, we obtain
\[
\kappa_{\infty}
\left(
M_{\mathcal S}
\right)
\leq
W_{\max}(\mathcal S)
L_{\max}(\mathcal S)
\kappa_{\infty}
\left(
M_{\mathcal E}
\right).
\]
Thus, for a fixed grid, weight, and polynomial basis, the dependence of
the bound on the interval family is determined by its maximum
combinatorial width and by the maximum distance between consecutive grid vertices in the endpoint tree.
\end{remark}

The preceding geometric representation also provides a characterization of
the spectral condition number.

\begin{corollary}
\label{cor:spectralConditionNumberPathMatrix}
Assume that $A_{\mathcal S}$ is nonsingular. Then
\begin{equation}
\label{eq:spectralConditionNumberPathMatrix}
\kappa_2\left(A_{\mathcal S}\right)
=
\kappa_2\left(B_{\mathcal S}\right)
=
\left[
\frac{
\lambda_{\max}\left(B_{\mathcal S}^{\top}B_{\mathcal S}\right)
}{
\lambda_{\min}\left(B_{\mathcal S}^{\top}B_{\mathcal S}\right)
}
\right]^{1/2},
\end{equation}
where $B_{\mathcal S}$ is defined in~\eqref{eq_BS}. Moreover,
\[
\kappa_2\left(A_{\mathcal S}\right)
\geq
\sqrt{
W_{\max}(\mathcal S)L_{\max}(\mathcal S)
}.
\]
\end{corollary}

\begin{proof}
By Theorem~\ref{thm:pathInverseIntervalMatrix}, we have
\begin{equation}\label{ASm1B}
    A_{\mathcal S}^{-1}=B_{\mathcal S}.
\end{equation}
Since the singular values of an inverse matrix are the reciprocals of the
singular values of the original matrix, we get
\[
\sigma_{\max}\left(A_{\mathcal S}^{-1}\right)
=
\frac{1}{\sigma_{\min}\left(A_{\mathcal S}\right)},
\quad
\sigma_{\min}\left(A_{\mathcal S}^{-1}\right)
=
\frac{1}{\sigma_{\max}\left(A_{\mathcal S}\right)}.
\]
Therefore, using~\eqref{ASm1B}, we obtain
\[
\kappa_2\left(B_{\mathcal S}\right)=\kappa_2\left(A_{\mathcal S}^{-1}\right)
=
\frac{
\sigma_{\max}\left(A_{\mathcal S}^{-1}\right)
}{
\sigma_{\min}\left(A_{\mathcal S}^{-1}\right)
}
=
\frac{
\sigma_{\max}\left(A_{\mathcal S}\right)
}{
\sigma_{\min}\left(A_{\mathcal S}\right)
}
=
\kappa_2\left(A_{\mathcal S}\right).
\]
Moreover, by the definition of the singular values,
\[
\sigma_j^2\left(B_{\mathcal S}\right)
=
\lambda_j\left(
B_{\mathcal S}^{\top}B_{\mathcal S}
\right),
\]
and consequently
\[
\kappa_2\left(B_{\mathcal S}\right)
=
\left[
\frac{
\lambda_{\max}\left(
B_{\mathcal S}^{\top}B_{\mathcal S}
\right)
}{
\lambda_{\min}\left(
B_{\mathcal S}^{\top}B_{\mathcal S}
\right)
}
\right]^{1/2}.
\]
This proves~\eqref{eq:spectralConditionNumberPathMatrix}. The $i$-th row of $A_{\mathcal S}$ contains exactly
$r_i-\ell_i$ entries equal to one and therefore has Euclidean norm
\[
\sqrt{r_i-\ell_i}.
\]
Since the spectral norm is at least the Euclidean norm of each row, we have
\[
\left\|A_{\mathcal S}\right\|_2
\geq
\sqrt{W_{\max}(\mathcal S)}.
\]
Similarly, the $j$-th row of $B_{\mathcal S}$ contains exactly
$\left|\mathcal P_j\right|$ nonzero entries, each equal to $1$ or $-1$,
and hence
\[
\left\|B_{\mathcal S}\right\|_2
\geq
\sqrt{L_{\max}(\mathcal S)}.
\]
Then
\[
\kappa_2\left(A_{\mathcal S}\right)
=
\left\|A_{\mathcal S}\right\|_2
\left\|B_{\mathcal S}\right\|_2
\geq
\sqrt{
W_{\max}(\mathcal S)L_{\max}(\mathcal S)
}.
\]
\end{proof}

We now illustrate the preceding results with two concrete families of intervals. 
In both cases, the condition number in the infinity norm can be computed explicitly, 
while the spectral condition number can be bounded from below using the estimate derived above.

\begin{example}
\label{ex:nestedIntervals}
Let $N\geq2$, and consider the family of nested intervals
\[
\mathcal S_{\mathrm{nest}}
=
\left\{
s_1,\ldots,s_N
\right\},
\quad
s_i=\left[x_0,x_i\right],
\quad
i=1,\ldots,N.
\]
For each $i=1,\ldots,N$, we have $\ell_i=0$ and $r_i=i$. Therefore,
\[
\left[A_{\mathcal S_{\mathrm{nest}}}\right]_{ij}
=
\begin{cases}
1, & j\leq i,\\
0, & j>i,
\end{cases}
\quad
i,j=1,\ldots,N,
\]
and hence
\[
A_{\mathcal S_{\mathrm{nest}}}
=
\begin{pmatrix}
1 & 0 & 0 & \cdots & 0\\
1 & 1 & 0 & \cdots & 0\\
1 & 1 & 1 & \cdots & 0\\
\vdots & \vdots & \vdots & \ddots & \vdots\\
1 & 1 & 1 & \cdots & 1
\end{pmatrix}.
\]
Thus, $A_{\mathcal S_{\mathrm{nest}}}$ is lower triangular with unit
diagonal and is therefore nonsingular. The associated endpoint graph has edge set
\[
\mathcal E_{\mathcal S_{\mathrm{nest}}}
=
\left\{
\{0,i\}
\,:\,
i=1,\ldots,N
\right\}.
\]
Hence, $\mathcal G_{\mathcal S_{\mathrm{nest}}}$ is a star centered at
the vertex $0$ and is therefore connected. Thus,
Theorem~\ref{thm:endpointTreeCriterion} provides an alternative proof
of the unisolvence of $\mathcal S_{\mathrm{nest}}$. We now determine the inverse of the interval matrix. For $j=1$, the
unique path $\mathcal P_1$ joining $0$ to $1$ is given by the sequence
\[
0,1.
\]
For $j=2,\ldots,N$, the unique path $\mathcal P_j$ joining $j-1$ to $j$
is given by the sequence
\[
j-1,0,j.
\]
Since each edge $\{0,i\}$ is oriented from $0$ to $i$, we have
\[
\varepsilon_{11}=1,
\]
whereas
\[
\varepsilon_{j,j-1}=-1,
\quad
\varepsilon_{j,j}=1,
\quad
j=2,\ldots,N,
\]
and all the remaining entries vanish. Therefore,
Theorem~\ref{thm:pathInverseIntervalMatrix} gives
\[
A_{\mathcal S_{\mathrm{nest}}}^{-1}
=
\begin{pmatrix}
1 & 0 & 0 & \cdots & 0\\
-1 & 1 & 0 & \cdots & 0\\
0 & -1 & 1 & \cdots & 0\\
\vdots & \vdots & \ddots & \ddots & \vdots\\
0 & 0 & \cdots & -1 & 1
\end{pmatrix}.
\]
Moreover, we have
\[
\left|\mathcal P_1\right|=1,
\quad
\left|\mathcal P_j\right|=2,
\quad
j=2,\ldots,N.
\]
Hence, by Corollary~\ref{cor:geometricNormsInverse}, we obtain
\[
\left\|
A_{\mathcal S_{\mathrm{nest}}}^{-1}
\right\|_{\infty}
=
2.
\]
Thus, the infinity norm of the inverse is independent of $N$. On the
other hand,
\[
\left\|
A_{\mathcal S_{\mathrm{nest}}}
\right\|_{\infty}
=
N,
\]
and therefore
\[
\kappa_{\infty}
\left(
A_{\mathcal S_{\mathrm{nest}}}
\right)
=
2N.
\]
The same conclusion also follows directly from
Corollary~\ref{cor:infinityConditionNumberFactorization}. Indeed,
\[
W_{\max}
\left(
\mathcal S_{\mathrm{nest}}
\right)
=
N,
\quad
L_{\max}
\left(
\mathcal S_{\mathrm{nest}}
\right)
=
2,
\]
and hence
\[
\kappa_{\infty}
\left(
A_{\mathcal S_{\mathrm{nest}}}
\right)
=
W_{\max}
\left(
\mathcal S_{\mathrm{nest}}
\right)
L_{\max}
\left(
\mathcal S_{\mathrm{nest}}
\right)
=
2N.
\]
Moreover, by Corollary~\ref{cor:spectralConditionNumberPathMatrix}, we get
\[
\kappa_2\left(A_{\mathcal S_{\mathrm{nest}}}\right)
\geq \sqrt{2N}.
\]
\end{example}

\begin{example}
\label{ex:anchoredSlidingWindow}
Let $N\geq3$, and consider the sliding-window family
\[
\mathcal S_{\mathrm{sw}}
=
\left\{
\left[x_i,x_{i+2}\right]
\,:\,
i=0,\ldots,N-2
\right\}.
\]
Each interval is the union of two consecutive elementary intervals. Since
$\mathcal S_{\mathrm{sw}}$ contains only $N-1$ intervals,
Theorem~\ref{thm:endpointTreeCriterion} cannot be applied directly. Its
endpoint graph, however, suggests a natural way to complete the family. The edge set of the endpoint graph is
\[
\mathcal E_{\mathcal S_{\mathrm{sw}}}
=
\left\{
\{i,i+2\}
\,:\,
i=0,\ldots,N-2
\right\}.
\]
These edges form two disjoint paths, one containing the even vertices and
the other containing the odd vertices. Hence,
$\mathcal G_{\mathcal S_{\mathrm{sw}}}$ is disconnected. We now add the auxiliary elementary interval $\left[x_0,x_1\right]$ and define
\[
\widehat{\mathcal S}_{\mathrm{sw}}
=
\left\{
\left[x_0,x_1\right]
\right\}
\cup
\left\{
\left[x_i,x_{i+2}\right]
\,:\,
i=0,\ldots,N-2
\right\}.
\]
The family $\widehat{\mathcal S}_{\mathrm{sw}}$ consists of $N$ pairwise
distinct intervals. Its endpoint graph is obtained from
$\mathcal G_{\mathcal S_{\mathrm{sw}}}$ by adding the edge $\{0,1\}$,
which joins the paths containing the even and odd vertices. Therefore,
$\mathcal G_{\widehat{\mathcal S}_{\mathrm{sw}}}$ is connected. It follows
from Theorem~\ref{thm:endpointTreeCriterion} that the interval matrix
$A_{\widehat{\mathcal S}_{\mathrm{sw}}}$ is nonsingular and that
$\widehat{\mathcal S}_{\mathrm{sw}}$ is unisolvent on $\Pi_{N-1}$. Thus, a single  auxiliary interval turns the sliding-window family into a unisolvent configuration, as illustrated in Fig.~\ref{fig:slidingWindow}.

We now determine the interval matrix of the supplemented family and its
condition number in the infinity norm. We set
\[
s_1=\left[x_0,x_1\right],
\quad
s_i=\left[x_{i-2},x_i\right],
\quad
i=2,\ldots,N.
\]
Then
\[
A_{\widehat{\mathcal S}_{\mathrm{sw}}}
=
\begin{pmatrix}
1 & 0 & 0 & \cdots & 0\\
1 & 1 & 0 & \cdots & 0\\
0 & 1 & 1 & \cdots & 0\\
\vdots & \ddots & \ddots & \ddots & \vdots\\
0 & \cdots & 0 & 1 & 1
\end{pmatrix}.
\]
Let
\[
D_N
=
\operatorname{diag}
\left(
1,-1,1,\ldots,(-1)^{N-1}
\right).
\]
Then
\[
D_N^{\top}=D_N^{-1}=D_N.
\]
Using the expression for
$A_{\mathcal S_{\mathrm{nest}}}^{-1}$ obtained in
Example~\ref{ex:nestedIntervals}, we have
\begin{equation}
\label{eq:nestedSlidingRelation}
A_{\widehat{\mathcal S}_{\mathrm{sw}}}
=
D_N
A_{\mathcal S_{\mathrm{nest}}}^{-1}
D_N.
\end{equation}
Equivalently,
\[
A_{\widehat{\mathcal S}_{\mathrm{sw}}}^{-1}
=
D_N
A_{\mathcal S_{\mathrm{nest}}}
D_N.
\]
Therefore,
\[
A_{\widehat{\mathcal S}_{\mathrm{sw}}}^{-1}
=
\begin{pmatrix}
1 & 0 & 0 & \cdots & 0\\
-1 & 1 & 0 & \cdots & 0\\
1 & -1 & 1 & \cdots & 0\\
\vdots & \vdots & \ddots & \ddots & \vdots\\
(-1)^{N-1} & (-1)^{N-2} & \cdots & -1 & 1
\end{pmatrix}.
\]
The $j$-th row of
$A_{\widehat{\mathcal S}_{\mathrm{sw}}}^{-1}$
has exactly $j$ nonzero entries. By
Theorem~\ref{thm:pathInverseIntervalMatrix}, these entries correspond to the
edges of the unique path $\mathcal P_j$. Hence,
\[
\left|\mathcal P_j\right|=j,
\quad
j=1,\ldots,N.
\]
Hence, by Corollary~\ref{cor:geometricNormsInverse},
\[
\left\|
A_{\widehat{\mathcal S}_{\mathrm{sw}}}^{-1}
\right\|_{\infty}
=
N.
\]
On the other hand,
\[
\left\|
A_{\widehat{\mathcal S}_{\mathrm{sw}}}
\right\|_{\infty}
=
2,
\]
and therefore
\[
\kappa_{\infty}
\left(
A_{\widehat{\mathcal S}_{\mathrm{sw}}}
\right)
=
2N.
\]
The same conclusion also follows from
Corollary~\ref{cor:infinityConditionNumberFactorization}. Indeed,
\[
W_{\max}
\left(
\widehat{\mathcal S}_{\mathrm{sw}}
\right)
=
2,
\quad
L_{\max}
\left(
\widehat{\mathcal S}_{\mathrm{sw}}
\right)
=
N,
\]
and hence
\[
\kappa_{\infty}
\left(
A_{\widehat{\mathcal S}_{\mathrm{sw}}}
\right)
=
W_{\max}
\left(
\widehat{\mathcal S}_{\mathrm{sw}}
\right)
L_{\max}
\left(
\widehat{\mathcal S}_{\mathrm{sw}}
\right)
=
2N.
\]
Consequently, 
Corollary~\ref{cor:spectralConditionNumberPathMatrix} gives
\[
\kappa_2\left(
A_{\widehat{\mathcal S}_{\mathrm{sw}}}
\right)
\geq
\sqrt{2N}.
\]
By symmetry, the same conclusions hold if
$\left[x_0,x_1\right]$ is replaced by
$\left[x_{N-1},x_N\right]$. Sliding-window constructions are standard in statistical applications~\cite{Datar:2002:MSS}.
\end{example}

\begin{figure}[ht]
  \centering
  \includegraphics[width=0.49\textwidth]{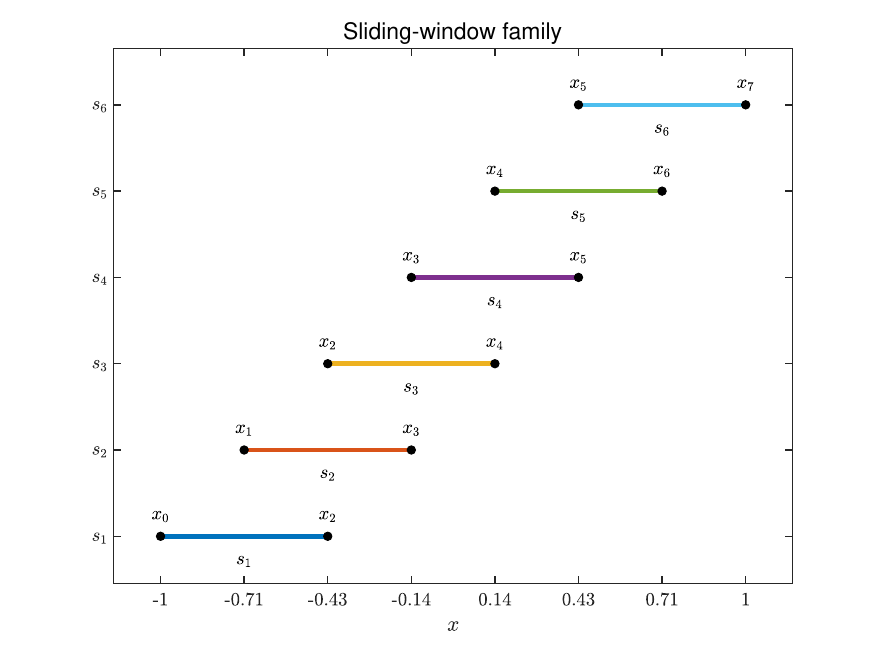}
  \hfill
  \includegraphics[width=0.49\textwidth]{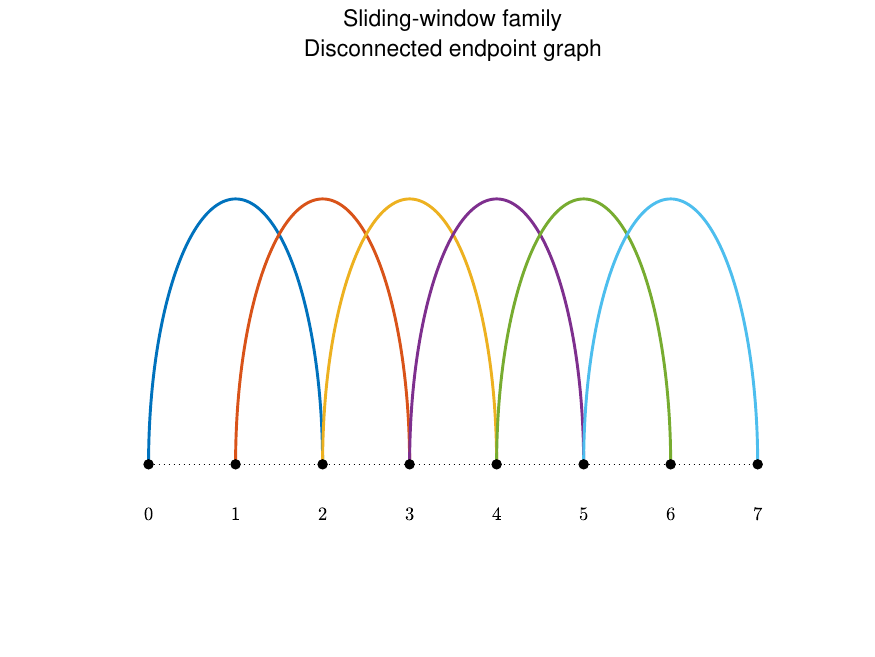}
    \includegraphics[width=0.49\textwidth]{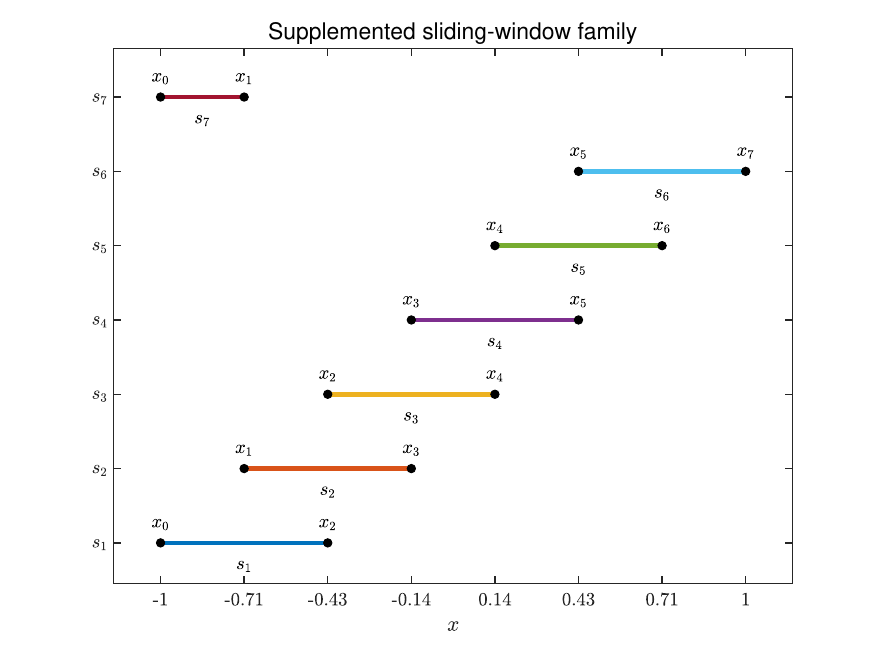}
  \hfill
  \includegraphics[width=0.49\textwidth]{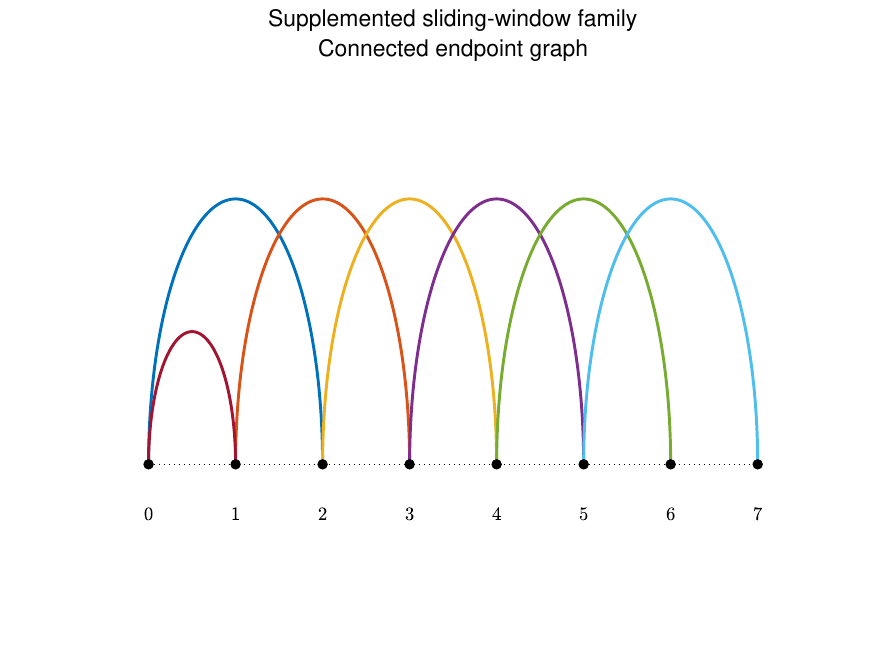}
\caption{Sliding-window configurations for $N=7$. Top: the family
$\mathcal S_{\mathrm{sw}}$ (left) and its endpoint graph (right).
Bottom: the supplemented family
$\widehat{\mathcal S}_{\mathrm{sw}}$ (left) and its endpoint graph (right).
The additional interval $\left[x_0,x_1\right]$ introduces the edge $\{0,1\}$, joins
the two components, and yields a connected graph.}
  \label{fig:slidingWindow}
\end{figure}

We now examine the spectral consequences of
relation~\eqref{eq:nestedSlidingRelation}. The matrix
$A_{\mathcal S_{\mathrm{nest}}}^{-1}$ is the reduced node-arc
incidence matrix of an oriented path with $N+1$ vertices. Since
multiplication by $D_N$ only changes the signs of rows and columns (in fact $D_N$ is a unitary matrix),
relation~\eqref{eq:nestedSlidingRelation} implies that
$A_{\widehat{\mathcal S}_{\mathrm{sw}}}$ has the same singular values
as this incidence matrix; see~\cite[Section~2.2]{Frangioni:2001:SAO}.
This observation yields explicit formulas for the singular values and
the spectral condition numbers of both matrices. We then study the
asymptotic singular value distributions of the corresponding matrix
sequences within the GLT
framework~\cite{Garoni:2017:GLT1,Garoni:2018:GLT2}.

\begin{proposition}
\label{prop:nestedSlidingSpectralProperties}
For any $N\geq3$, the singular values of
$A_{\widehat{\mathcal S}_{\mathrm{sw}}}$, up to ordering, are
\[
2\sin
\left(
\frac{(2j-1)\pi}{4N+2}
\right),
\quad
j=1,\ldots,N,
\]
whereas the singular values of
$A_{\mathcal S_{\mathrm{nest}}}$ are
\[
\frac{1}{
2\sin
\left(
\frac{(2j-1)\pi}{4N+2}
\right)},
\quad
j=1,\ldots,N.
\]
Consequently,
\begin{equation}\label{eq:statedformula}
    \kappa_2
\left(
A_{\mathcal S_{\mathrm{nest}}}
\right)
=
\kappa_2
\left(
A_{\widehat{\mathcal S}_{\mathrm{sw}}}
\right)
=
\frac{
\cos\left(\dfrac{\pi}{2N+1}\right)
}{
\sin\left(\dfrac{\pi}{4N+2}\right)
},
\end{equation}
and
\[
\kappa_2
\left(
A_{\mathcal S_{\mathrm{nest}}}
\right)
=
\kappa_2
\left(
A_{\widehat{\mathcal S}_{\mathrm{sw}}}
\right)
\sim
\frac{4}{\pi}N, \quad  N\to\infty.
\]
\end{proposition}

\begin{proof}
Since $D_N$ is unitary, relation~\eqref{eq:nestedSlidingRelation} shows that
$A_{\widehat{\mathcal S}_{\mathrm{sw}}}$ and
$A_{\mathcal S_{\mathrm{nest}}}^{-1}$ have the same singular values.
Therefore, the singular values of
$A_{\mathcal S_{\mathrm{nest}}}$ are the reciprocals of those of
$A_{\widehat{\mathcal S}_{\mathrm{sw}}}$. A direct computation gives
\[
A_{\widehat{\mathcal S}_{\mathrm{sw}}}^{\top}
A_{\widehat{\mathcal S}_{\mathrm{sw}}}
=
\begin{pmatrix}
2 & 1 & 0 & \cdots & 0\\
1 & 2 & 1 & \ddots & \vdots\\
0 & 1 & 2 & \ddots & 0\\
\vdots & \ddots & \ddots & \ddots & 1\\
0 & \cdots & 0 & 1 & 1
\end{pmatrix}.
\]
For $j=1,\ldots,N$, set
\[
\theta_j
=
\frac{(2j-1)\pi}{2N+1}
\]
and define
\[
\boldsymbol v^{(j)}
=
\left(
v_1^{(j)},\ldots,v_N^{(j)}
\right)^{\top},
\quad
v_k^{(j)}
=
(-1)^{k-1}\sin\left(k\theta_j\right).
\]
Using
\[
\sin\left((k-1)\theta\right)
+
\sin\left((k+1)\theta\right)
=
2\cos(\theta)\sin(k\theta),
\]
together with
\[
\sin\left((N+1)\theta_j\right)
=
\sin\left(N\theta_j\right),
\]
straightforward computations give
\[
A_{\widehat{\mathcal S}_{\mathrm{sw}}}^{\top}
A_{\widehat{\mathcal S}_{\mathrm{sw}}}
\boldsymbol v^{(j)}
=
\left(
2-2\cos\left(\theta_j\right)
\right)
\boldsymbol v^{(j)}.
\]
Thus, for each $j=1,\ldots,N$,
the vector $\boldsymbol v^{(j)}$ is an eigenvector of
$A_{\widehat{\mathcal S}_{\mathrm{sw}}}^{\top}
A_{\widehat{\mathcal S}_{\mathrm{sw}}}$
associated with the eigenvalue
\[
\lambda_j
=
2-2\cos\left(\theta_j\right)
=
4\sin^2
\left(
\frac{(2j-1)\pi}{4N+2}
\right).
\]
Since
\begin{equation*}
    0<\theta_1<\cdots<\theta_N<\pi,
\end{equation*}
the eigenvalues $\lambda_1,\ldots,\lambda_N$ are pairwise distinct.
They therefore constitute the full spectrum of
$A_{\widehat{\mathcal S}_{\mathrm{sw}}}^{\top}
A_{\widehat{\mathcal S}_{\mathrm{sw}}}$. Moreover,
\[
0<
\frac{\theta_1}{2}
<
\cdots
<
\frac{\theta_N}{2}
<
\frac{\pi}{2}.
\]
Therefore, since the sine function is strictly increasing on
\[
\left(0,\frac{\pi}{2}\right),
\]
we obtain
\[
\sigma_{\min}
\left(
A_{\widehat{\mathcal S}_{\mathrm{sw}}}
\right)
=
2\sin
\left(
\frac{\theta_1}{2}
\right)
=
2\sin
\left(
\frac{\pi}{4N+2}
\right),
\]
and
\[
\sigma_{\max}
\left(
A_{\widehat{\mathcal S}_{\mathrm{sw}}}
\right)
=
2\sin
\left(
\frac{\theta_N}{2}
\right)
=
2\sin
\left(
\frac{(2N-1)\pi}{4N+2}
\right)
=
2\cos
\left(
\frac{\pi}{2N+1}
\right).
\]
This gives the stated formula~\eqref{eq:statedformula}.  Moreover, we have
\[
\sin
\left(
\frac{\pi}{4N+2}
\right)
\sim
\frac{\pi}{4N},
\quad
\lim_{N\to \infty}\cos
\left(
\frac{\pi}{2N+1}
\right)=1,
\]
and therefore
\[
\kappa_2
\left(
A_{\mathcal S_{\mathrm{nest}}}
\right)
=
\kappa_2
\left(
A_{\widehat{\mathcal S}_{\mathrm{sw}}}
\right)
\sim
\frac{4}{\pi}N.
\]
\end{proof}

\begin{remark}
The preceding matrix sequences also admit an explicit description of their singular value distributions in terms of their GLT symbols.  Indeed,
$A_{\widehat{\mathcal S}_{\mathrm{sw}}}$ is the Toeplitz matrix
generated by
\[
f_1(\theta)=1+e^{\mathrm{i}\theta},
\quad \theta\in[-\pi,\pi],
\]
and therefore
\[
\left\{
A_{\widehat{\mathcal S}_{\mathrm{sw}}}
\right\}_{N}
\sim_{\mathrm{GLT}}
1+e^{\mathrm{i}\theta}.
\]
Consequently,
\[
\left\{
A_{\widehat{\mathcal S}_{\mathrm{sw}}}
\right\}_{N}
\sim_{\sigma}
1+e^{\mathrm{i}\theta}.
\]
Moreover,
$A_{\mathcal S_{\mathrm{nest}}}^{-1}$ is the Toeplitz matrix
generated by
\[
f_2(\theta)=1-e^{\mathrm{i}\theta}, \quad \theta\in[-\pi,\pi],
\]
so that
\[
\left\{
A_{\mathcal S_{\mathrm{nest}}}^{-1}
\right\}_{N}
\sim_{\mathrm{GLT}}
1-e^{\mathrm{i}\theta}.
\]
Since $1-e^{\mathrm{i}\theta}$ vanishes only at $\theta=0$, it is
nonzero almost everywhere. Hence, by the GLT inversion property, we have
\[
\left\{
A_{\mathcal S_{\mathrm{nest}}}
\right\}_{N}
\sim_{\mathrm{GLT}}
\frac{1}{1-e^{\mathrm{i}\theta}},
\]
and therefore
\[
\left\{
A_{\mathcal S_{\mathrm{nest}}}
\right\}_{N}
\sim_{\sigma}
\frac{1}{1-e^{\mathrm{i}\theta}}.
\]
The corresponding nonnegative singular value symbols are
\[
\left|1+e^{\mathrm{i}\theta}\right|
=
2\left|\cos\left(\frac{\theta}{2}\right)\right|
\]
and
\[
\left|
\frac{1}{1-e^{\mathrm{i}\theta}}
\right|
=
\frac{1}{
2\left|\sin\left(\frac{\theta}{2}\right)\right|},
\]
respectively.
\end{remark}

\section{Exactly diagonal weighted histopolation configurations}\label{sec3}

In this section, we present a weighted histopolation configuration for which
the associated Gram matrix is exactly diagonal. We focus on the Chebyshev
weight of the first kind, namely
\[
\alpha=\beta=-\frac{1}{2},
\quad
\omega_{-1/2,-1/2}(x)=\frac{1}{\sqrt{1-x^2}}.
\]
Using the classical Jacobi normalization~\cite{Milovanovic:1997:OPS}, we have
\begin{equation}
\label{ChebJacobiNormalization}
P_k^{(-1/2,-1/2)}(\cos\theta)
=
a_k\cos(k\theta),
\quad
a_k=\frac{(1/2)_k}{k!},
\quad
k\geq0,
\end{equation}
where, for a given number $s$, $(s)_k$ denotes the Pochhammer symbol for the rising factorial. Let $N\geq1$ and set
\begin{equation}
\label{rhodef}
0<\rho<\frac{\pi}{2N},
\quad
\tau_i=\frac{(2i-1)\pi}{2N},
\quad
i=1,\ldots,N.
\end{equation}
We define the cells of constant angular length
\begin{equation}
\label{constantAngularCells}
s_i=
\left[\cos\left(\tau_i+\rho\right),\cos\left(\tau_i-\rho\right)\right],
\quad
i=1,\ldots,N.
\end{equation}
The restriction in~\eqref{rhodef} guarantees that all the intervals
in~\eqref{constantAngularCells} are contained in $(-1,1)$. For each cell,
we define its weighted mass by
\begin{equation*}
m_i=
\int_{\cos\left(\tau_i+\rho\right)}^{\cos\left(\tau_i-\rho\right)}
\omega_{-1/2,-1/2}(x)dx.
\end{equation*}
The weighted Chebyshev histopolation matrix
$H_N\in\mathbb R^{N\times N}$ is then defined by
\begin{equation}
\label{HNchdef}
\left[H_N\right]_{i,k+1}
=
\frac{1}{m_i}
\int_{\cos\left(\tau_i+\rho\right)}^{\cos\left(\tau_i-\rho\right)}
P_k^{(-1/2,-1/2)}(x)\omega_{-1/2,-1/2}(x)dx,
\end{equation}
for $i=1,\ldots,N$ and $k=0,\ldots,N-1$.

We first recall the following standard discrete cosine orthogonality
relations; see, for instance, \cite[Chapter~4]{Mason:2002:CP}.

\begin{lemma}\label{lemmaCosineDiscreteOrthogonality}
Let $\tau_i$, $i=1,\ldots,N$, be defined by~\eqref{rhodef}. Then
\begin{equation}\label{cosineSumZero}
\sum_{i=1}^N \cos\left(m\tau_i\right)=0, \quad m=1,\dots,2N-1.
\end{equation}
Consequently,
\begin{equation}\label{cosineDiscreteOrthogonality}
\sum_{i=1}^N\cos\left(k\tau_i\right)\cos\left(\ell\tau_i\right)
=
\begin{cases}
0, & k\neq \ell,\\
N/2, & k=\ell,
\end{cases}
\end{equation}
for any $1\le k,\ell\le N-1.$
\end{lemma}

We now prove that, for the above choice of cells, the Gram matrix associated
with $H_N$ is diagonal.

\begin{theorem}
\label{thmChebFirstKindDiagonalGram}
For any $0<\rho< \frac{\pi}{2N}$, the matrix $H_N$ defined in~\eqref{HNchdef} satisfies
\begin{equation}\label{HNchGram}
\left(H_N\right)^{\top} H_N
=
\operatorname{diag}\left(\mu_0,\mu_1,\ldots,\mu_{N-1}\right),
\end{equation}
where
\begin{equation*}
\mu_0=N, \quad \mu_k
=
\frac{N}{2}
\left[
a_k\frac{\sin(k\rho)}{k\rho}
\right]^2, \quad k=1,\dots,N-1.
\end{equation*}
In particular, $H_N$ is nonsingular.
\end{theorem}

\begin{proof}
We first compute the entries of $H_N$. Using the change of
variables $x=\cos\theta$, we have
\[
dx=-\sin\theta d\theta,
\quad
\omega_{-1/2,-1/2}(\cos\theta)=\frac{1}{\sin\theta},
\quad 0<\theta<\pi.
\]
Hence
\begin{equation*}
    m_i
=
\int_{\cos\left(\tau_i+\rho\right)}^{\cos\left(\tau_i-\rho\right)}
\omega_{-1/2,-1/2}(x)dx =
\int_{\tau_i+\rho}^{\tau_i-\rho}(-d\theta)
=
2\rho.
\end{equation*}
Since
\[
P_0^{(-1/2,-1/2)}=1,
\]
the first column of $H_N$ is constant, that is
\begin{equation}\label{HN11}
\left[H_N\right]_{i,1}=1,
\quad i=1,\ldots,N.
\end{equation}
For $k=1,\ldots,N-1$, using~\eqref{ChebJacobiNormalization} together with the identity
\[
\sin\left(k\left(\tau_i+\rho\right)\right)-\sin\left(k\left(\tau_i-\rho\right)\right)=2\sin(k\rho)\cos\left(k\tau_i\right),
\]
we get
\begin{eqnarray}
\left[H_N\right]_{i,k+1} &=& \frac{1}{2\rho}
\int_{\cos\left(\tau_i+\rho\right)}^{\cos\left(\tau_i-\rho\right)}
P_k^{(-1/2,-1/2)}(x)\omega_{-1/2,-1/2}(x)dx \notag\\
&=&
\frac{a_k}{2\rho}
\int_{\tau_i-\rho}^{\tau_i+\rho}\cos(k\theta)d\theta \notag\\
&=&
\frac{a_k}{2k\rho}
\left[
\sin\left(k\left(\tau_i+\rho\right)\right)-\sin\left(k\left(\tau_i-\rho\right)\right)
\right] \notag\\
&=&
a_k\frac{\sin(k\rho)}{k\rho}\cos\left(k\tau_i\right).
\label{HNchEntries}
\end{eqnarray}
We now compute the entries of the Gram matrix $\left(H_N\right)^{\top} H_N$. From~\eqref{HN11}, we have
\[
\left[\left(H_N\right)^{\top} H_N\right]_{1,1}
=\sum_{i=1}^N \left[\left(H_N\right)^{\top}\right]_{1,i}\left[H_N\right]_{i,1}=
\sum_{i=1}^N1=N.
\]
For any $k=1,\ldots,N-1$, combining~\eqref{HNchEntries} with Lemma~\ref{lemmaCosineDiscreteOrthogonality},
we obtain
\begin{equation*}
\left[\left(H_N\right)^{\top} H_N\right]_{1,k+1}
=\sum_{i=1}^N \left[\left(H_N\right)^{\top}\right]_{1,i}\left[H_N\right]_{i,k+1}=
a_k\frac{\sin(k\rho)}{k\rho}
\sum_{i=1}^N\cos\left(k\tau_i\right)
=
0.
\end{equation*}
Similarly, for any $k,\ell=1,\ldots,N-1$, we get
\begin{eqnarray*}
\left[\left(H_N\right)^{\top} H_N\right]_{k+1,\ell+1}
&=&
a_k a_\ell
\frac{\sin(k\rho)}{k\rho}
\frac{\sin(\ell\rho)}{\ell\rho}
\sum_{i=1}^N\cos\left(k\tau_i\right)\cos\left(\ell\tau_i\right)\\&=&
\begin{cases}
0, & k\neq \ell,\\
\frac{N}{2}
\left[
a_k\frac{\sin(k\rho)}{k\rho}
\right]^2, & k=\ell,
\end{cases}
\end{eqnarray*}
which proves~\eqref{HNchGram}. It remains to show that all the diagonal
entries are positive. Since
\[
\mu_0=N>0,
\]
we only need to consider $k=1,\ldots,N-1$. By~\eqref{rhodef}, for such
values of $k$, we have
\[
0<k\rho
<
\frac{k\pi}{2N}
\leq
\frac{(N-1)\pi}{2N}
<
\frac{\pi}{2}.
\]
Thus $\sin(k\rho)>0$, and therefore
\[
\mu_k>0,
\quad
k=1,\ldots,N-1.
\]
Hence $\left(H_N\right)^{\top}H_N$ is positive definite, and consequently $H_N$
is nonsingular.
\end{proof}

\begin{remark}
For $N=1$, we have $H_1=[1]$, and hence $\kappa_2\left(H_1\right)=1$. Assume now
that $N\geq2$. By Theorem~\ref{thmChebFirstKindDiagonalGram}, the
singular values of $H_N$, up to ordering, are
\begin{equation*}
\sigma_0\left(H_N\right)=\sqrt{N},
\quad
\sigma_k\left(H_N\right)
=
\sqrt{\frac{N}{2}}
a_k\frac{\sin(k\rho)}{k\rho},
\quad
k=1,\ldots,N-1.
\end{equation*}
For any $k\geq0$, using~\eqref{ChebJacobiNormalization}, we have
\begin{equation*}
\frac{a_{k+1}}{a_k}
=
\frac{(1/2)_{k+1}}{(k+1)!}
\frac{k!}{(1/2)_k}=
\frac{k+1/2}{k+1}
<1.
\end{equation*}
Hence, the sequence $\left\{a_k\right\}_{k\geq0}$ is strictly decreasing. Moreover, by~\eqref{rhodef}, for $k=1,\ldots,N-1$, we have
\[
0<k\rho<\frac{\pi}{2}.
\]
The function
\[
x\mapsto\frac{\sin x}{x}
\]
is positive and strictly decreasing on $(0,\pi)$. Since the sequence
$\left\{a_k\right\}_{k\geq0}$ is also positive and strictly decreasing, it follows
that
\[
a_k\frac{\sin(k\rho)}{k\rho},
\quad
k=1,\ldots,N-1,
\]
is strictly decreasing. Moreover, $a_1=1/2$ and
$\sin(\rho)/\rho<1$, and therefore
\[
\sigma_1\left(H_N\right)
=
\sqrt{\frac{N}{2}}
\frac{\sin(\rho)}{2\rho}
<
\sqrt{N}
=
\sigma_0\left(H_N\right).
\]
Therefore, for $N\geq2$, we get
\begin{equation*}
\sigma_{\min}\left(H_N\right)
=
\sqrt{\frac{N}{2}}
a_{N-1}
\frac{\sin((N-1)\rho)}{(N-1)\rho}
\end{equation*}
and
\begin{equation*}
\sigma_{\max}\left(H_N\right)=\sqrt{N}.
\end{equation*}
Consequently,
\begin{equation*}
\kappa_2\left(H_N\right)
=
\frac{\sigma_{\max}\left(H_N\right)}
{\sigma_{\min}\left(H_N\right)}
=
\frac{\sqrt{2}}{a_{N-1}}
\frac{(N-1)\rho}{\sin((N-1)\rho)}.
\end{equation*}
Since the function
\[
x\mapsto\frac{x}{\sin x}
\]
is strictly increasing on $(0,\pi)$, the map
\[
\rho\mapsto\kappa_2\left(H_N\right)
\]
is strictly increasing on $(0,\pi/(2N))$. Thus, within the admissible
range~\eqref{rhodef}, smaller angular cells yield better conditioning.
In particular,
\begin{equation*}
\inf_{0<\rho<\frac{\pi}{2N}}\kappa_2\left(H_N\right)
=
\lim_{\rho\to0^+}\kappa_2\left(H_N\right)
=
\frac{\sqrt{2}}{a_{N-1}}.
\end{equation*}
\end{remark}

We now examine the role of the Jacobi parameters in the construction above.
Throughout this part, we consider the Jacobi weight
\[
\omega_{\alpha,\beta}(x)
=
(1-x)^\alpha(1+x)^\beta,
\quad
\alpha,\beta>-1.
\]
Let $\tau_i$ and $s_i$ be defined as in~\eqref{rhodef}
and~\eqref{constantAngularCells}, respectively. We define
\[
m_i^{(\alpha,\beta)}
=
\int_{\cos\left(\tau_i+\rho\right)}^{\cos\left(\tau_i-\rho\right)}
\omega_{\alpha,\beta}(x)dx
\]
and
\[
\left[H_N^{(\alpha,\beta)}\right]_{i,k+1}
=
\frac{1}{m_i^{(\alpha,\beta)}}
\int_{\cos\left(\tau_i+\rho\right)}^{\cos\left(\tau_i-\rho\right)}
P_k^{(\alpha,\beta)}(x)\omega_{\alpha,\beta}(x)dx,
\]
for $i=1,\ldots,N$ and $k=0,\ldots,N-1$. The next result
characterizes the Jacobi parameters for which the diagonal structure
of the Gram matrix persists for every admissible value of $\rho$.

\begin{theorem}
Let $\alpha,\beta>-1$ and let $N\geq3$. Then the following statements are
equivalent:
\begin{enumerate}
\item For any $\rho\in\left(0,\frac{\pi}{2N}\right)$, the Gram matrix
\begin{equation}
\label{grammat}
\left(H_N^{(\alpha,\beta)}\right)^{\top}H_N^{(\alpha,\beta)}
\end{equation}
is diagonal.
\item $\alpha=\beta=-\frac{1}{2}$.
\end{enumerate}
\end{theorem}

\begin{proof}
If $\alpha=\beta=-1/2$, then the diagonality statement follows from
Theorem~\ref{thmChebFirstKindDiagonalGram}. It remains to prove the other
implication.

Assume that the Gram matrix~\eqref{grammat} is diagonal for any
\[
0<\rho<\frac{\pi}{2N}.
\]
Using the change of variables $x=\cos\theta$, we obtain
\[
m_i^{(\alpha,\beta)}
=
\int_{\tau_i-\rho}^{\tau_i+\rho}
\omega_{\alpha,\beta}(\cos\theta)\sin\theta d\theta
\]
and
\begin{eqnarray}
\left[H_N^{(\alpha,\beta)}\right]_{i,k+1}&=&\frac{1}{m_i^{(\alpha,\beta)}}
\int_{\cos\left(\tau_i+\rho\right)}^{\cos\left(\tau_i-\rho\right)}
P_k^{(\alpha,\beta)}(x)\omega_{\alpha,\beta}(x)dx \notag\\
&=&
\frac{1}{m_i^{(\alpha,\beta)}}
\int_{\tau_i-\rho}^{\tau_i+\rho}
P_k^{(\alpha,\beta)}(\cos\theta)
\omega_{\alpha,\beta}(\cos\theta)\sin\theta d\theta, \label{19HN}
\end{eqnarray}
for $i=1,\ldots,N$ and $k=0,\ldots,N-1$. In particular, since
$P_0^{(\alpha,\beta)}=1$, we have
\begin{equation}
\label{HNi1}
\left[H_N^{(\alpha,\beta)}\right]_{i,1}=1,
\quad
i=1,\ldots,N.
\end{equation}
For fixed $i=1,\ldots,N$ and $k=0,\ldots,N-1$, subtracting $P_k^{(\alpha,\beta)}\left(\cos\tau_i\right)$ from both sides of~\eqref{19HN} and using the definition of $m_i^{(\alpha,\beta)}$, we obtain
\begin{eqnarray*}
&&\left[H_N^{(\alpha,\beta)}\right]_{i,k+1}
-
P_k^{(\alpha,\beta)}\left(\cos\tau_i\right) \\
&=&\frac{1}{m_i^{(\alpha,\beta)}}
\int_{\tau_i-\rho}^{\tau_i+\rho}
\left[
P_k^{(\alpha,\beta)}(\cos\theta)
-
P_k^{(\alpha,\beta)}\left(\cos\tau_i\right)
\right]
\omega_{\alpha,\beta}(\cos\theta)\sin\theta d\theta.
\end{eqnarray*}
Therefore, by the triangle inequality and the positivity of the weight,
we have
\begin{eqnarray*}
&&\left|
\left[H_N^{(\alpha,\beta)}\right]_{i,k+1}
-
P_k^{(\alpha,\beta)}\left(\cos\tau_i\right)
\right|\\&=& \frac{1}{m_i^{(\alpha,\beta)}}\left|
\int_{\tau_i-\rho}^{\tau_i+\rho}
\left[
P_k^{(\alpha,\beta)}(\cos\theta)
-
P_k^{(\alpha,\beta)}\left(\cos\tau_i\right)
\right]
\omega_{\alpha,\beta}(\cos\theta)\sin\theta d\theta\right| \\
&\le& \frac{1}{m_i^{(\alpha,\beta)}}
\int_{\tau_i-\rho}^{\tau_i+\rho}
\left|
P_k^{(\alpha,\beta)}(\cos\theta)
-
P_k^{(\alpha,\beta)}\left(\cos\tau_i\right)
\right|
\omega_{\alpha,\beta}(\cos\theta)\sin\theta d\theta\\
&\le& \max_{\theta\in[\tau_i-\rho,\tau_i+\rho]}
\left|
P_k^{(\alpha,\beta)}(\cos\theta)
-
P_k^{(\alpha,\beta)}\left(\cos\tau_i\right)
\right|.
\end{eqnarray*}
Since the function
\[
\theta\mapsto P_k^{(\alpha,\beta)}(\cos\theta)
\]
is continuous at $\theta=\tau_i$, it follows that
\begin{equation}\label{JacobiHistLimit}
\lim_{\rho\to0^+}
\left[H_N^{(\alpha,\beta)}\right]_{i,k+1}
=
P_k^{(\alpha,\beta)}\left(\cos\tau_i\right).
\end{equation}
Using~\eqref{HNi1} and the assumption that the Gram matrix is diagonal,
we have
\[
0
=
\left[
\left(H_N^{(\alpha,\beta)}\right)^{\top}
H_N^{(\alpha,\beta)}
\right]_{1,2}
=
\sum_{i=1}^N \left[H_N^{(\alpha,\beta)}\right]_{i,2}.
\]
Since this identity holds for any $\rho\in(0,\pi/(2N))$, passing to the
limit as $\rho\to0^+$ and using~\eqref{JacobiHistLimit}, we get
\begin{equation}\label{0p1}
    0
=\lim_{\rho\to 0^+} \sum_{i=1}^N \left[H_N^{(\alpha,\beta)}\right]_{i,2}=
\sum_{i=1}^N P_1^{(\alpha,\beta)}\left(\cos\tau_i\right).
\end{equation}
By the classical Jacobi normalization~\cite{Milovanovic:1997:OPS}, we have
\begin{equation}\label{JacPol1}
    P_1^{(\alpha,\beta)}(x)
=
\frac{1}{2}\left[(\alpha+\beta+2)x+\alpha-\beta\right].
\end{equation}
Combining~\eqref{0p1} and~\eqref{JacPol1}, we obtain
\begin{equation}\label{newnew}
    0
=
\frac{\alpha+\beta+2}{2}
\sum_{i=1}^N\cos\tau_i
+
\frac{N}{2}(\alpha-\beta).
\end{equation}
Using~\eqref{cosineSumZero}, applied with $m=1$,
\[
\sum_{i=1}^N\cos\tau_i=0.
\]
Therefore~\eqref{newnew} reduces to
\[
\alpha-\beta=0,
\]
and hence $\alpha=\beta$.
Set
\[
\gamma:=\alpha=\beta.
\]
Since $N\ge3$, we can compute
\[
0
=
\left[
\left(H_N^{(\gamma,\gamma)}\right)^{\top}
H_N^{(\gamma,\gamma)}
\right]_{1,3}
=
\sum_{i=1}^N \left[H_N^{(\gamma,\gamma)}\right]_{i,3}.
\]
Passing again to the limit $\rho\to0^+$ and using~\eqref{JacobiHistLimit}, we obtain
\begin{equation*}
0=\lim_{\rho\to0^+}\sum_{i=1}^N \left[H_N^{(\gamma,\gamma)}\right]_{i,3}
=
\sum_{i=1}^N P_2^{(\gamma,\gamma)}\left(\cos\tau_i\right).
\end{equation*}
With the classical Jacobi normalization~\cite{Milovanovic:1997:OPS}, we have
\[
P_2^{(\gamma,\gamma)}(x)
=
\frac{\gamma+2}{4}
\left[(2\gamma+3)x^2-1\right].
\]
Hence
\begin{equation}
    \label{newnewnew}
    0
=
\frac{\gamma+2}{4}
\left[
(2\gamma+3)\sum_{i=1}^N\cos^2\tau_i
-
N
\right].
\end{equation}
Using~\eqref{cosineDiscreteOrthogonality}, with $k=\ell=1$, we have
\[
\sum_{i=1}^N\cos^2\tau_i=\frac{N}{2}.
\]
Therefore~\eqref{newnewnew} becomes
\[
0
=
\frac{ \gamma+2}{4}
\left[
(2\gamma+3)\frac{N}{2}-N
\right]
=
\frac{N(\gamma+2)(2\gamma+1)}{8}.
\]
Since $\gamma>-1$, we have
\[
2\gamma+1=0,
\]
that is
\[
\gamma=-\frac{1}{2}.
\]
Therefore $\alpha=\beta=-1/2$.
\end{proof}

\begin{remark}
The exactly diagonal configuration above also gives a construction in dimension
$d\geq2$ on Cartesian product cells. Let
\[
\boldsymbol{N}=(N_1,\ldots,N_d),
\quad
N_r\geq1,
\quad
r=1,\ldots,d,
\]
and, for each $r=1,\ldots,d$, let
\[
0<\rho_r<\frac{\pi}{2N_r}.
\]
We consider
\[
\Omega=[-1,1]^d
\]
with the product weight
\[
\omega^{(d)}(\boldsymbol{x})
=
\prod_{r=1}^d
\omega_{-1/2,-1/2}(x_r),
\quad
\boldsymbol{x}=(x_1,\ldots,x_d).
\]
For each coordinate $r$, set
\[
\tau_{i_r}^{(r)}
=
\frac{(2i_r-1)\pi}{2N_r},
\quad
i_r=1,\ldots,N_r,
\]
and define
\[
s_{i_r}^{(r)}
=
\left[
\cos\left(\tau_{i_r}^{(r)}+\rho_r\right),
\cos\left(\tau_{i_r}^{(r)}-\rho_r\right)
\right].
\]
We then consider the cells
\[
s_{\boldsymbol{i}}
=
s_{i_1}^{(1)}
\times\cdots\times
s_{i_d}^{(d)},
\quad
\boldsymbol{i}=(i_1,\ldots,i_d),
\]
and the polynomial basis
\[
P_{\boldsymbol{k}}(\boldsymbol{x})
=
\prod_{r=1}^d
P_{k_r}^{(-1/2,-1/2)}(x_r),
\quad
0\leq k_r\leq N_r-1,
\quad
r=1,\ldots,d.
\]
Since the cells, the weight, and the polynomial basis have a product form, we have
\[
H_{\boldsymbol{N}}^{(d)}
=
H_{N_1}\otimes\cdots\otimes H_{N_d},
\]
where $H_{N_r}$ is the one-dimensional normalized histopolation matrix
corresponding to $N_r$ and $\rho_r$. Therefore, by
Theorem~\ref{thmChebFirstKindDiagonalGram}, we have
\begin{eqnarray*}
\left(H_{\boldsymbol{N}}^{(d)}\right)^{\top}
H_{\boldsymbol{N}}^{(d)}
&=&
\bigotimes_{r=1}^d
\left(H_{N_r}^{\top}H_{N_r}\right)\\
&=&
\bigotimes_{r=1}^d
\operatorname{diag}
\left(
\mu_0^{(r)},\ldots,\mu_{N_r-1}^{(r)}
\right).
\end{eqnarray*}
Hence the Gram matrix is diagonal, with diagonal entries
\[
\prod_{r=1}^d\mu_{k_r}^{(r)},
\quad
0\leq k_r\leq N_r-1,
\quad
r=1,\ldots,d.
\]
Consequently, the singular values of $H_{\boldsymbol{N}}^{(d)}$, up to
ordering, are
\[
\prod_{r=1}^d\sqrt{\mu_{k_r}^{(r)}},
\quad
0\leq k_r\leq N_r-1,
\quad
r=1,\ldots,d,
\]
and
\[
\kappa_2\left(H_{\boldsymbol{N}}^{(d)}\right)
=
\prod_{r=1}^d
\kappa_2\left(H_{N_r}\right).
\]
Thus the exact diagonality and the explicit formulas for the singular values
extend to Cartesian product cells in any fixed dimension. The case of cells
which are not Cartesian products is not considered here.
\end{remark}

\section{A general diagonalization criterion}
\label{sec:diagonalizationCriterion}

In this section, we present a general framework for constructing
histopolation matrices with diagonal Gram matrices.

Let $N\geq2$ and let
\[
\Xi_N=\left\{\xi_1,\ldots,\xi_N\right\}
\]
be a set of distinct nodes. We consider a positive diagonal matrix
\begin{equation}
\label{matWn}
W_N=\operatorname{diag}\left(w_1,\ldots,w_N\right),
\quad
w_i>0,
\quad
i=1,\ldots,N,
\end{equation}
and real-valued functions
\[
\psi_0,\ldots,\psi_{N-1}
\]
defined on $\Xi_N$. We introduce the sampling matrix
\[
\Psi_N
=
\left[
\psi_k\left(\xi_i\right)
\right]_{\substack{i=1,\ldots,N\\ k=0,\ldots,N-1}}.
\]
We assume that the columns of $\Psi_N$ are orthogonal with respect to
the inner product induced by $W_N$, namely
\begin{equation}
\label{eq:generalWeightedOrthogonality}
\Psi_N^{\top}W_N\Psi_N
=
\Delta_N,
\end{equation}
where
\[
\Delta_N
=
\operatorname{diag}\left(\nu_0,\ldots,\nu_{N-1}\right),
\quad
\nu_k>0,
\quad
k=0,\ldots,N-1.
\]
Classical discrete cosine and sine transforms provide a natural source
of sampling matrices satisfying~\eqref{eq:generalWeightedOrthogonality}.
In particular, the four discrete cosine transforms of types I--IV and
the four discrete sine transforms of types I--IV, discussed in
\cite[Tables~I--III]{Benedetto:2000:OMM}, correspond to different
choices of sampling grids and to integer or half-integer shifts in the
node and frequency indices. With suitable normalization factors, the
associated transform matrices are real orthogonal. The criterion
developed below shows how this discrete orthogonality can be used to
construct histopolation matrices whose columns are nonzero scalar
multiples of the corresponding columns of the sampling matrix.

For $\boldsymbol u,\boldsymbol v\in\mathbb R^N$, we set
\[
\left\langle
\boldsymbol u,\boldsymbol v
\right\rangle_{W_N}
:=
\boldsymbol u^{\top}W_N\boldsymbol v.
\]
Equivalently, if
\begin{equation*}
\boldsymbol u_k
=
\left(
\psi_k\left(\xi_1\right),\ldots,\psi_k\left(\xi_N\right)
\right)^{\top},
\quad
k=0,\ldots,N-1,
\end{equation*}
then
\begin{equation}
\label{diagwq}
\left\langle
\boldsymbol u_j,\boldsymbol u_k
\right\rangle_{W_N}
=
\begin{cases}
0, & j\neq k,\\
\nu_k, & j=k,
\end{cases}
\quad
j,k=0,\ldots,N-1.
\end{equation}
Since $\Delta_N$ is positive diagonal, the matrix $\Psi_N$ is nonsingular.

Let $\omega$ be a nonnegative measurable weight on $(-1,1)$. We consider
measurable cells
\[
s_1,\ldots,s_N\subset(-1,1)
\]
such that
\[
0<
\int_{s_i}\omega(x)dx
<\infty,
\quad
i=1,\ldots,N.
\]
Let
\[
P_0,\ldots,P_{N-1}
\]
be real-valued functions that are integrable with respect to $\omega$
on each cell. We define the corresponding unnormalized histopolation matrix
$M_N\in\mathbb R^{N\times N}$ by
\begin{equation*}
\left[M_N\right]_{i,k+1}
=
\int_{s_i}P_k(x)\omega(x)dx,
\quad
i=1,\ldots,N,
\quad
k=0,\ldots,N-1.
\end{equation*}

The following proposition gives a sufficient condition for
$M_N^{\top}W_NM_N$ to be diagonal.

\begin{proposition}
\label{prop:columnwiseTransformMatching}
Assume that there exist nonzero real numbers $\gamma_0,\ldots,\gamma_{N-1}$ such that
\begin{equation}
\label{eq:columnwiseTransformMatching}
\int_{s_i}P_k(x)\omega(x)dx
=
\gamma_k\psi_k\left(\xi_i\right),
\quad
i=1,\ldots,N,
\quad
k=0,\ldots,N-1.
\end{equation}
Then
\begin{equation}
\label{eq:columnwiseDiagonalGram}
M_N^{\top}W_NM_N
=
\operatorname{diag}
\left(
\nu_0\gamma_0^2,\ldots,\nu_{N-1}\gamma_{N-1}^2
\right).
\end{equation}
In particular, $M_N$ is nonsingular.
\end{proposition}
\begin{proof}
By~\eqref{eq:columnwiseTransformMatching}, we have
\[
M_N
=
\Psi_N\Gamma_N,
\quad
\Gamma_N
=
\operatorname{diag}\left(\gamma_0,\ldots,\gamma_{N-1}\right).
\]
Using~\eqref{eq:generalWeightedOrthogonality}, we obtain
\begin{eqnarray*}
M_N^{\top}W_NM_N
&=&
\left(\Psi_N\Gamma_N\right)^{\top}W_N\left(\Psi_N\Gamma_N\right)\\
&=&
\Gamma_N^{\top}\left(\Psi_N^{\top}W_N\Psi_N\right)\Gamma_N\\
&=&
\Gamma_N^{\top}\Delta_N\Gamma_N\\
&=&
\operatorname{diag}
\left(
\nu_0\gamma_0^2,\ldots,\nu_{N-1}\gamma_{N-1}^2
\right).
\end{eqnarray*}
This proves~\eqref{eq:columnwiseDiagonalGram}. Since $\nu_k>0$ and
$\gamma_k\neq0$ for all $k=0,\ldots,N-1$, the matrix
$M_N^{\top}W_NM_N$ is positive definite. Consequently, $M_N$ is
nonsingular.
\end{proof}

In the remainder of this section, assume that
\[
P_0(x)=1.
\]
Fix an index
\[
m\in\{1,\ldots,N-1\}.
\]
Let
\[
\boldsymbol c_0
=
\left(
\int_{s_1}\omega(x)dx,\ldots,
\int_{s_N}\omega(x)dx
\right)^{\top},
\]
and, for $k=1,\ldots,N-1$, let
\[
\boldsymbol c_k
=
\left(
\int_{s_1}P_k(x)\omega(x)dx,\ldots,
\int_{s_N}P_k(x)\omega(x)dx
\right)^{\top}.
\]
Since $W_N$ is positive definite and $\boldsymbol c_0\neq\boldsymbol 0$,
we may define
\begin{equation}
\label{eq:generalLambda}
\lambda_m=\frac{\left\langle
\boldsymbol c_0,\boldsymbol c_m\right\rangle_{W_N}}{\left\langle
\boldsymbol c_0,\boldsymbol c_0\right\rangle_{W_N}}
=
\frac{
\boldsymbol c_0^{\top}W_N\boldsymbol c_m
}
{
\boldsymbol c_0^{\top}W_N\boldsymbol c_0
}.
\end{equation}
We consider the modified functions
\[
\widehat P_0(x)=1,
\quad
\widehat P_m(x)=P_m(x)-\lambda_m,
\]
and
\[
\widehat P_k(x)=P_k(x),
\quad
k\in\{1,\ldots,N-1\}\setminus\{m\}.
\]
Let $\hat{M}_N\in\mathbb R^{N\times N}$ be the corresponding
unnormalized histopolation matrix, namely
\begin{equation}
\label{eq:modifiedMomentMatrix}
\left[\hat{M}_N\right]_{i,k+1}
=
\int_{s_i}\widehat P_k(x)\omega(x)dx,
\quad
i=1,\ldots,N,
\quad
k=0,\ldots,N-1.
\end{equation}
The columns of $\hat{M}_N$ are
\begin{equation}
\label{newbasis1}
\widehat{\boldsymbol c}_0
=
\boldsymbol c_0,
\quad
\widehat{\boldsymbol c}_m
=
\boldsymbol c_m
-
\lambda_m\boldsymbol c_0,
\end{equation}
and
\begin{equation}
\label{akequation}
\widehat{\boldsymbol c}_k
=
\boldsymbol c_k,
\quad
k\in\{1,\ldots,N-1\}\setminus\{m\}.
\end{equation}
Suppose that there exist real numbers
\[
\eta_0\neq0,
\quad
\eta_m,
\quad
\widehat a_k\neq0,
\quad
k=1,\ldots,N-1,
\]
such that
\begin{equation}
\label{eq:generalConstantColumnReduction}
\boldsymbol c_0
=
\eta_0\boldsymbol u_0+\eta_m\boldsymbol u_m
\end{equation}
and
\begin{equation}
\label{eq:generalNonconstantColumnReduction}
\boldsymbol c_k
=
\widehat a_k\boldsymbol u_k,
\quad
k=1,\ldots,N-1.
\end{equation}
The next theorem shows that correcting a single column in this way yields
a diagonal weighted Gram matrix.

\begin{theorem}
\label{thm:singleColumnCorrection}
Let $\hat{M}_N$ and $W_N$ be defined
in~\eqref{eq:modifiedMomentMatrix} and~\eqref{matWn}, respectively.
Assume that~\eqref{eq:generalConstantColumnReduction}
and~\eqref{eq:generalNonconstantColumnReduction} hold. Then
\begin{equation}
\label{eq:generalOneModeDiagonalGram}
\hat{M}_N^{\top}W_N\hat{M}_N
=
\operatorname{diag}
\left(
d_0,\ldots,d_{N-1}
\right),
\end{equation}
where
\begin{equation*}
d_0
=
\mu_0
=
\eta_0^2\nu_0+\eta_m^2\nu_m,
\quad
d_m
=
\widehat\mu_m
=
\frac{
\widehat a_m^2\nu_0\nu_m\eta_0^2
}
{
\eta_0^2\nu_0+\eta_m^2\nu_m
},
\end{equation*}
and
\begin{equation*}
d_k
=
\mu_k
=
\widehat a_k^2\nu_k,
\quad
k\in\{1,\ldots,N-1\}\setminus\{m\}.
\end{equation*}
In particular, $\hat{M}_N$ is nonsingular.
\end{theorem}

\begin{proof}
By~\eqref{diagwq}, we have
\[
\left\langle
\boldsymbol u_j,\boldsymbol u_k
\right\rangle_{W_N}=\boldsymbol u_j^{\top}W_N\boldsymbol u_k
=
\begin{cases}
0, & j\neq k,\\
\nu_k, & j=k,
\end{cases}
\quad
j,k=0,\ldots,N-1.
\]
Using~\eqref{eq:generalConstantColumnReduction}, we obtain
\begin{eqnarray}
\left\langle
\boldsymbol c_0,\boldsymbol c_0\right\rangle_{W_N}
&=&
\left(
\eta_0\boldsymbol u_0+\eta_m\boldsymbol u_m
\right)^{\top}
W_N
\left(
\eta_0\boldsymbol u_0+\eta_m\boldsymbol u_m
\right)
\notag\\
&=&
\eta_0^2\nu_0+\eta_m^2\nu_m.
\label{eq:generalMuZeroProof}
\end{eqnarray}
Moreover, by~\eqref{eq:generalConstantColumnReduction}
and~\eqref{eq:generalNonconstantColumnReduction}, we have
\begin{equation}
\label{eq:generalMixedProduct}
\left\langle
\boldsymbol c_0,\boldsymbol c_k
\right\rangle_{W_N}
=
\begin{cases}
0, & k\neq m,\\
\eta_m\widehat a_m\nu_m, & k=m,
\end{cases}
\quad
k=1,\ldots,N-1.
\end{equation}
Therefore, by~\eqref{eq:generalLambda}, we get
\begin{equation}
\label{newlambda}
\lambda_m
=
\frac{
\eta_m\widehat a_m\nu_m
}
{
\eta_0^2\nu_0+\eta_m^2\nu_m
}.
\end{equation}
We now compute the scalar products between the columns of
$\hat{M}_N$. By~\eqref{newbasis1} and~\eqref{eq:generalLambda}, we have
\begin{eqnarray*}
\left\langle
\widehat{\boldsymbol c}_0,\widehat{\boldsymbol c}_m
\right\rangle_{W_N}=\widehat{\boldsymbol c}_0^{\top}
W_N
\widehat{\boldsymbol c}_m
&=&
\boldsymbol c_0^{\top}
W_N
\left(
\boldsymbol c_m-\lambda_m\boldsymbol c_0
\right)\\
&=&
\boldsymbol c_0^{\top}W_N\boldsymbol c_m
-
\lambda_m
\boldsymbol c_0^{\top}W_N\boldsymbol c_0\\
&=&
0.
\end{eqnarray*}
For $k\in\{1,\ldots,N-1\}\setminus\{m\}$,
using~\eqref{eq:generalMixedProduct}, we obtain
\[
\left\langle
\widehat{\boldsymbol c}_0,\widehat{\boldsymbol c}_k
\right\rangle_{W_N}=\widehat{\boldsymbol c}_0^{\top}
W_N
\widehat{\boldsymbol c}_k
=
\boldsymbol c_0^{\top}W_N\boldsymbol c_k
=
0.
\]
Similarly, using~\eqref{eq:generalMixedProduct}
and~\eqref{eq:generalNonconstantColumnReduction}, we have
\begin{eqnarray*}
\left\langle
\widehat{\boldsymbol c}_m,\widehat{\boldsymbol c}_k
\right\rangle_{W_N}
&=&
\left(
\boldsymbol c_m-\lambda_m\boldsymbol c_0
\right)^{\top}
W_N
\boldsymbol c_k\\
&=&
\boldsymbol c_m^{\top}W_N\boldsymbol c_k
-
\lambda_m
\boldsymbol c_0^{\top}W_N\boldsymbol c_k\\
&=&
\widehat a_m\widehat a_k
\boldsymbol u_m^{\top}W_N\boldsymbol u_k\\
&=&
0.
\end{eqnarray*}
For the remaining off-diagonal entries, namely for
\[
j,k\in\{1,\ldots,N-1\}\setminus\{m\},
\quad
j\neq k,
\]
using~\eqref{akequation}
and~\eqref{eq:generalNonconstantColumnReduction}, we obtain
\[
\left\langle
\widehat{\boldsymbol c}_j,\widehat{\boldsymbol c}_k
\right\rangle_{W_N}=\widehat{\boldsymbol c}_j^{\top}
W_N
\widehat{\boldsymbol c}_k
=
\boldsymbol c_j^{\top}W_N\boldsymbol c_k
=
\widehat a_j\widehat a_k
\boldsymbol u_j^{\top}W_N\boldsymbol u_k
=
0.
\]
Thus all the off-diagonal entries of
$\hat{M}_N^{\top}W_N\hat{M}_N$ vanish.

It remains to compute the diagonal entries. First,
by~\eqref{eq:generalMuZeroProof}, we get
\[
\left\langle
\widehat{\boldsymbol c}_0,\widehat{\boldsymbol c}_0
\right\rangle_{W_N}=\widehat{\boldsymbol c}_0^{\top}
W_N
\widehat{\boldsymbol c}_0
=
\boldsymbol c_0^{\top}W_N\boldsymbol c_0
=
\eta_0^2\nu_0+\eta_m^2\nu_m.
\]
For
\[
k\in\{1,\ldots,N-1\}\setminus\{m\},
\]
we have
\[
\left\langle
\widehat{\boldsymbol c}_k,\widehat{\boldsymbol c}_k
\right\rangle_{W_N}=\widehat{\boldsymbol c}_k^{\top}
W_N
\widehat{\boldsymbol c}_k
=
\widehat a_k^2
\boldsymbol u_k^{\top}W_N\boldsymbol u_k
=
\widehat a_k^2\nu_k.
\]
Finally, by~\eqref{eq:generalMuZeroProof},
\eqref{eq:generalMixedProduct}, and~\eqref{newlambda}, we obtain
\begin{eqnarray*}
\left\langle
\widehat{\boldsymbol c}_m,\widehat{\boldsymbol c}_m
\right\rangle_{W_N}
&=&
\left(
\boldsymbol c_m-\lambda_m\boldsymbol c_0
\right)^{\top}
W_N
\left(
\boldsymbol c_m-\lambda_m\boldsymbol c_0
\right)\\
&=&
\boldsymbol c_m^{\top}W_N\boldsymbol c_m
-
\frac{
\left(
\boldsymbol c_0^{\top}W_N\boldsymbol c_m
\right)^2
}
{
\boldsymbol c_0^{\top}W_N\boldsymbol c_0
}\\
&=&
\widehat a_m^2\nu_m
-
\frac{
\eta_m^2\widehat a_m^2\nu_m^2
}
{
\eta_0^2\nu_0+\eta_m^2\nu_m
}\\
&=&
\frac{
\widehat a_m^2\nu_0\nu_m\eta_0^2
}
{
\eta_0^2\nu_0+\eta_m^2\nu_m
}.
\end{eqnarray*}
Since $\eta_0\neq0$, $\widehat a_k\neq0$, and $\nu_k>0$, all the
diagonal entries are positive. Therefore,
$\hat{M}_N^{\top}W_N\hat{M}_N$ is positive definite, and hence
$\hat{M}_N$ is nonsingular.
\end{proof}

\begin{remark}
The value $\lambda_m$ in~\eqref{eq:generalLambda} is uniquely determined
by the requirement that the modified $m$-th column be orthogonal to the
first one. Indeed, for any $\lambda\in\mathbb R$, replacing $P_m$ by
$P_m-\lambda$ gives the column
\[
\boldsymbol c_m-\lambda\boldsymbol c_0.
\]
Its scalar product with $\boldsymbol c_0$ is
\[
\boldsymbol c_0^{\top}W_N
\left(
\boldsymbol c_m-\lambda\boldsymbol c_0
\right)
=
\boldsymbol c_0^{\top}W_N\boldsymbol c_m
-
\lambda
\boldsymbol c_0^{\top}W_N\boldsymbol c_0.
\]
Since
\[
\boldsymbol c_0^{\top}W_N\boldsymbol c_0>0,
\]
this scalar product vanishes if and only if
\[
\lambda
=
\frac{
\boldsymbol c_0^{\top}W_N\boldsymbol c_m
}
{
\boldsymbol c_0^{\top}W_N\boldsymbol c_0
}
=\frac{\left\langle
\boldsymbol c_0,\boldsymbol c_m
\right\rangle_{W_N}}{\left\langle
\boldsymbol c_0,\boldsymbol c_0
\right\rangle_{W_N}}=
\lambda_m.
\]
All the remaining off-diagonal scalar products vanish independently of
$\lambda$. Therefore, $\lambda_m$ is the unique value for which the
replacement
\[
P_m(x)\to P_m(x)-\lambda
\]
yields a diagonal weighted Gram matrix.
\end{remark}

\begin{remark}
The correction of the $m$-th column can be written as
\[
\widehat{\boldsymbol c}_m
=
\boldsymbol c_m
-
\frac{\left\langle
\boldsymbol c_0,\boldsymbol c_m
\right\rangle_{W_N}}{\left\langle
\boldsymbol c_0,\boldsymbol c_0
\right\rangle_{W_N}}
\boldsymbol c_0.
\]
Thus, $\widehat{\boldsymbol c}_m$ is obtained by subtracting from
$\boldsymbol c_m$ its $W_N$-orthogonal projection onto
$\operatorname{span}\left\{\boldsymbol c_0\right\}$.
\end{remark}

\begin{remark}
The proof of Theorem~\ref{thm:singleColumnCorrection} uses only the symmetry and
positive definiteness of the inner product induced by $W_N$, together with
the orthogonality relation~\eqref{eq:generalWeightedOrthogonality}.
Consequently, the same conclusion remains valid when $W_N$ is an arbitrary
symmetric positive definite matrix.
\end{remark}

\begin{remark}
The diagonal Gram identity in
Theorem~\ref{thm:singleColumnCorrection} also gives the singular values
of the weighted matrix $W_N^{1/2}\hat M_N$, up to ordering. Indeed,
since $W_N$ is positive diagonal, we have
\[
\left(
W_N^{1/2}\hat M_N
\right)^{\top}
\left(
W_N^{1/2}\hat M_N
\right)
=
\hat M_N^{\top}W_N\hat M_N.
\]
Therefore, by~\eqref{eq:generalOneModeDiagonalGram}, the singular values
of $W_N^{1/2}\hat M_N$ are
\[
\sqrt{\mu_0},
\quad
\sqrt{\widehat\mu_m},
\quad
\sqrt{\mu_k},
\quad
k\in\{1,\ldots,N-1\}\setminus\{m\},
\]
up to ordering. Consequently,
\[
\kappa_2\left(W_N^{1/2}\hat M_N\right)
=
\left(
\frac{
\max\left(
\left\{\mu_0,\widehat\mu_m\right\}
\cup
\left\{\mu_k:\ 1\leq k\leq N-1,\ k\neq m\right\}
\right)
}{
\min\left(
\left\{\mu_0,\widehat\mu_m\right\}
\cup
\left\{\mu_k:\ 1\leq k\leq N-1,\ k\neq m\right\}
\right)
}
\right)^{1/2}.
\]
In the special case $W_N=I_N$, these are the usual singular values and
the spectral condition number of $\hat M_N$.
\end{remark}

\begin{remark}
For $N\geq2$, Theorem~\ref{thmChebFirstKindDiagonalGram} can be recovered
as a special case of Theorem~\ref{thm:singleColumnCorrection}. Let
\[
W_N=I_N,
\quad
\xi_i=\tau_i=\frac{(2i-1)\pi}{2N},
\quad
i=1,\ldots,N,
\]
and consider the functions
\[
\psi_0(\theta)=1,
\quad
\psi_k(\theta)=\cos(k\theta),
\quad
k=1,\ldots,N-1.
\]
The corresponding sampled vectors are
\[
\boldsymbol u_0=(1,\ldots,1)^{\top}
\]
and
\[
\boldsymbol u_k
=
\left(
\cos\left(k\xi_1\right),\ldots,\cos\left(k\xi_N\right)
\right)^{\top},
\quad
k=1,\ldots,N-1.
\]
By Lemma~\ref{lemmaCosineDiscreteOrthogonality}, we have
\[
\boldsymbol u_j^{\top}\boldsymbol u_k
=
\frac{N}{2}
\left(
1+\delta_{k0}
\right)
\delta_{jk},
\quad
j,k=0,\ldots,N-1,
\]
where $\delta_{jk}$ denotes the Kronecker delta. Thus,
\[
\nu_0=N,
\quad
\nu_k=\frac{N}{2},
\quad
k=1,\ldots,N-1.
\]
For the first kind Chebyshev weight, all the cell masses are equal to
$2\rho$. Hence, the columns of the corresponding unnormalized moment
matrix $M_N$ satisfy
\[
\boldsymbol c_0
=
2\rho\boldsymbol u_0
\]
and, by~\eqref{HNchEntries},
\[
\boldsymbol c_k
=
\frac{2a_k\sin(k\rho)}{k}\boldsymbol u_k,
\quad
k=1,\ldots,N-1,
\]
where
\[
a_k=\frac{(1/2)_k}{k!}.
\]
Therefore, in the notation of
Theorem~\ref{thm:singleColumnCorrection}, we may take, for instance,
\[
m=1,
\quad
\eta_0=2\rho,
\quad
\eta_1=0,
\]
and
\[
\widehat a_k
=
\frac{2a_k\sin(k\rho)}{k},
\quad
k=1,\ldots,N-1.
\]
Since $\eta_1=0$, formula~\eqref{newlambda} gives
\[
\lambda_1=0.
\]
Thus, no correction is required and $\hat M_N=M_N$. The diagonal entries
of $M_N^{\top}M_N$ are therefore
\[
d_0=4N\rho^2
\]
and
\[
d_k
=
\frac{N}{2}
\left[
\frac{2a_k\sin(k\rho)}{k}
\right]^2,
\quad
k=1,\ldots,N-1.
\]
Since all the cell masses are equal to $2\rho$, the normalized
histopolation matrix is
\[
H_N=\frac{1}{2\rho}M_N.
\]
Consequently,
\[
H_N^{\top}H_N
=
\frac{1}{4\rho^2}M_N^{\top}M_N.
\]
\end{remark}
We now apply Proposition~\ref{prop:columnwiseTransformMatching} to a
configuration based on a discrete sine system. In particular, we consider
\[
W_N=I_N,
\quad
\xi_i=\frac{i\pi}{N+1},
\quad
i=1,\ldots,N,
\]
together with the functions
\[
\psi_{k-1}(\theta)=\sin(k\theta),
\quad
k=1,\ldots,N.
\]
We choose the cells and the polynomial basis so that the columns
of the associated histopolation matrix are nonzero scalar multiples of the
sampled vectors
\[
\boldsymbol u_{k-1}
=
\left(
\sin\left(k\xi_1\right),\ldots,\sin\left(k\xi_N\right)
\right)^{\top},
\quad
k=1,\ldots,N.
\]
Let
\[
0<\rho<\frac{\pi}{N+1},
\]
and define
\[
s_i
=
[\cos(\xi_i+\rho),\cos(\xi_i-\rho)],
\quad
i=1,\ldots,N.
\]
The restriction on $\rho$ ensures that
\[
0<\xi_i-\rho<\xi_i+\rho<\pi,
\quad
i=1,\ldots,N,
\]
and hence that each cell $s_i$ is contained in $(-1,1)$. Let $U_k$ denote the Chebyshev polynomial of the second kind,
characterized by
\[
U_k(\cos\theta)\sin\theta
=
\sin((k+1)\theta),
\quad
0<\theta<\pi.
\]
We define the corresponding unnormalized histopolation matrix
$M_N\in\mathbb R^{N\times N}$ by
\begin{equation}
\label{M_Nsecondtype}
\left[M_N\right]_{i,k}
=
\int_{s_i}U_{k-1}(x)dx,
\quad
i,k=1,\ldots,N.
\end{equation}

We also recall the following standard discrete sine orthogonality
relations; see, for instance, \cite[Chapter~4]{Mason:2002:CP}.

\begin{lemma}
\label{lemmaSineDiscreteOrthogonality}
For the nodes defined above, we have
\begin{equation*}
\boldsymbol u_{k-1}^{\top}\boldsymbol u_{\ell-1}
=
\sum_{i=1}^N
\sin(k\xi_i)\sin(\ell\xi_i)
=
\begin{cases}
0, & k\neq \ell,\\
(N+1)/2, & k=\ell,
\end{cases}
\end{equation*}
for $k,\ell=1,\ldots,N.$
\end{lemma}

The construction considered here is related to the unweighted
histopolation problem studied in~\cite{Bruni:2025:OTC}. There, the
Chebyshev polynomials of the second kind are integrated over cells of
constant angular length centered at
\[
\tau_i=\frac{(2i-1)\pi}{2N},
\quad
i=1,\ldots,N.
\]
The resulting unnormalized histopolation matrix has orthogonal columns,
which yields explicit expressions for its Gram matrix and spectral
condition number. The configuration considered below retains both the
polynomial basis and the angular parametrization of the cells, while
replacing the points $\tau_i$ by
\[
\xi_i=\frac{i\pi}{N+1},
\quad
i=1,\ldots,N.
\]

\begin{proposition}
\label{prop:SineHistopolation}
The unnormalized histopolation matrix $M_N$ defined
in~\eqref{M_Nsecondtype} satisfies
\begin{equation}
\label{eq:SineEntries}
\left[M_N\right]_{i,k}
=
\frac{2\sin(k\rho)}{k}\sin(k\xi_i),
\quad
i,k=1,\ldots,N.
\end{equation}
Consequently,
\begin{equation}
\label{eq:SineGram}
M_N^{\top}M_N
=
\operatorname{diag}
\left(
\frac{N+1}{2}
\left[
\frac{2\sin(k\rho)}{k}
\right]^2
\right)_{k=1}^N.
\end{equation}
In particular, $M_N$ is nonsingular.
\end{proposition}

\begin{proof}
Using the change of variables $x=\cos\theta$, we have
$dx=-\sin\theta d\theta$. Therefore,
\begin{eqnarray*}
\left[M_N\right]_{i,k}
&=&
\int_{\cos(\xi_i+\rho)}^{\cos(\xi_i-\rho)}
U_{k-1}(x) dx=
\int_{\xi_i-\rho}^{\xi_i+\rho}
U_{k-1}(\cos\theta)\sin\theta d\theta\\
&=&
\int_{\xi_i-\rho}^{\xi_i+\rho}
\sin(k\theta) d\theta=
\frac{2\sin(k\rho)}{k}\sin(k\xi_i),
\end{eqnarray*}
which proves~\eqref{eq:SineEntries}.

In the notation of
Proposition~\ref{prop:columnwiseTransformMatching}, we have
\[
W_N=I_N,
\quad
\psi_{k-1}(\theta)=\sin(k\theta),
\quad
\gamma_{k-1}=\frac{2\sin(k\rho)}{k},
\quad
k=1,\ldots,N.
\]
Moreover, Lemma~\ref{lemmaSineDiscreteOrthogonality} gives
\[
\nu_{k-1}=\frac{N+1}{2},
\quad
k=1,\ldots,N.
\]
Since
\[
0<k\rho
\leq N\rho
<
\frac{N\pi}{N+1}
<
\pi,
\quad
k=1,\ldots,N,
\]
we have $\sin(k\rho)>0$, and hence $\gamma_{k-1}\neq0$ for any
$k=1,\ldots,N$. Therefore, by Proposition~\ref{prop:columnwiseTransformMatching}, we have
\[
M_N^{\top}M_N
=
\operatorname{diag}
\left(
\frac{N+1}{2}
\left[
\frac{2\sin(k\rho)}{k}
\right]^2
\right)_{k=1}^N,
\]
which proves~\eqref{eq:SineGram}. Hence
$M_N$ is nonsingular.
\end{proof}

\subsection{Orthogonalization associated with the interval graph}
\label{subsec:intervalGraphOrthogonalization}

We conclude this section with a construction based on the
topological criterion of Section~\ref{sec:combinatorialUnisolvence}.

Let
\[
X_N=\left\{x_0,\ldots,x_N\right\},
\quad
-1=x_0<x_1<\cdots<x_N=1,
\]
and let
\[
e_i=\left[x_{i-1},x_i\right],
\quad
i=1,\ldots,N,
\]
denote the elementary intervals associated with the grid. Let
\[
\mathcal S=\left\{s_1,\ldots,s_N\right\},
\quad
s_i=\left[x_{\ell_i},x_{r_i}\right],
\quad
i=1,\ldots,N,
\]
be a family of intervals whose endpoints belong to the grid. We denote by
\[
\mathcal G_{\mathcal S}
=
\left(\mathcal V_N,\mathcal E_{\mathcal S}\right)
\]
the associated endpoint graph, where
\[
\mathcal V_N=\{0,\ldots,N\},
\quad
\mathcal E_{\mathcal S}
=
\left\{
\{\ell_i,r_i\}\, :\, i=1,\ldots,N
\right\}.
\]
For
\[
\zeta\in\{-1,1\},
\]
we set
\[
\omega_{\alpha,\beta}^{(\zeta)}(x)
=
(1-\zeta x)\omega_{\alpha,\beta}(x)
=
\begin{cases}
\omega_{\alpha+1,\beta}(x), & \zeta=1,\\
\omega_{\alpha,\beta+1}(x), & \zeta=-1,
\end{cases}
\quad
\alpha,\beta>-1.
\]
This weight is positive almost everywhere in $(-1,1)$. For
$p\in\Pi_{N-1}$, we define the corresponding moment vector
\[
\boldsymbol m_{\mathcal S}^{(\zeta)}(p)
=
\left(
\int_{s_1}p(x)\omega_{\alpha,\beta}^{(\zeta)}(x)dx,
\ldots,
\int_{s_N}p(x)\omega_{\alpha,\beta}^{(\zeta)}(x)dx
\right)^{\top}.
\]
Let
\[
W_N=\operatorname{diag}\left(w_1,\ldots,w_N\right),
\quad
w_i>0,
\quad
i=1,\ldots,N.
\]
We introduce the bilinear form
\begin{equation}
\label{eq:intervalGraphInnerProduct}
[p,q]_{\mathcal S,W_N}^{(\zeta)}
=
\left(
\boldsymbol m_{\mathcal S}^{(\zeta)}(p)
\right)^{\top}
W_N
\boldsymbol m_{\mathcal S}^{(\zeta)}(q),
\quad
p,q\in\Pi_{N-1}.
\end{equation}

\begin{proposition}
\label{prop:intervalGraphInnerProduct}
The following statements are equivalent:
\begin{enumerate}
\item $\mathcal G_{\mathcal S}$ is connected;
\item $[\cdot,\cdot]_{\mathcal S,W_N}^{(\zeta)}$ is an inner product on
$\Pi_{N-1}$.
\end{enumerate}
\end{proposition}

\begin{proof}
The bilinear form defined in~\eqref{eq:intervalGraphInnerProduct} is
symmetric. Hence, to prove that it is an inner product, it remains only to
characterize its positive definiteness.

Let
\begin{equation}\label{basisa}
    P_0,\ldots,P_{N-1}, \quad \deg\left(P_k\right)=k, \quad k=0,\dots,N-1
\end{equation}
be any basis of $\Pi_{N-1}$. We denote by
$M_{\mathcal E}^{(\zeta)}$ and $M_{\mathcal S}^{(\zeta)}$ the matrices defined by
\begin{eqnarray*}
\left[M_{\mathcal E}^{(\zeta)}\right]_{i,k+1}
&=&
\int_{e_i}P_k(x)\omega_{\alpha,\beta}^{(\zeta)}(x)dx, \quad  i=1,\ldots,N,
\quad
k=0,\ldots,N-1 \\
\left[M_{\mathcal S}^{(\zeta)}\right]_{i,k+1}
&=&
\int_{s_i}P_k(x)\omega_{\alpha,\beta}^{(\zeta)}(x)dx, \quad
    i=1,\ldots,N,
\quad
k=0,\ldots,N-1.
\end{eqnarray*}
By Proposition~\ref{prop:subgridIntervalFactorization}, we have
\begin{equation}
\label{eq:intervalGraphFactorization}
M_{\mathcal S}^{(\zeta)}
=
A_{\mathcal S}M_{\mathcal E}^{(\zeta)}.
\end{equation}
Since $\omega_{\alpha,\beta}^{(\zeta)}$ is positive almost everywhere in
$(-1,1)$, the elementary histopolation problem is unisolvent on
$\Pi_{N-1}$. Hence $M_{\mathcal E}^{(\zeta)}$ is nonsingular.

Assume first that $\mathcal G_{\mathcal S}$ is connected. By
Theorem~\ref{thm:endpointTreeCriterion}, the matrix $A_{\mathcal S}$ is
nonsingular. Therefore, by~\eqref{eq:intervalGraphFactorization}, we have
\[
\det\left(M_{\mathcal S}^{(\zeta)}\right)
=
\det\left(A_{\mathcal S}\right)\det\left(M_{\mathcal E}^{(\zeta)}\right)
\neq 0.
\]
Hence $M_{\mathcal S}^{(\zeta)}$ is nonsingular. Let
$p\in\Pi_{N-1}$ be nonzero and write it with respect to the
basis~\eqref{basisa} as
\[
p=\sum_{k=0}^{N-1}a_kP_k,
\quad
\boldsymbol a=\left(a_0,\ldots,a_{N-1}\right)^{\top}.
\]
Since $p\neq0$, we have $\boldsymbol a\neq\boldsymbol 0$. Moreover,
\[
\boldsymbol m_{\mathcal S}^{(\zeta)}(p)
=
M_{\mathcal S}^{(\zeta)}\boldsymbol a.
\]
Since $M_{\mathcal S}^{(\zeta)}$ is nonsingular, it follows that
\[
\boldsymbol m_{\mathcal S}^{(\zeta)}(p)\neq\boldsymbol 0.
\]
Using the positive definiteness of $W_N$, we obtain
\[
[p,p]_{\mathcal S,W_N}^{(\zeta)}
=
\left(
\boldsymbol m_{\mathcal S}^{(\zeta)}(p)
\right)^{\top}
W_N
\boldsymbol m_{\mathcal S}^{(\zeta)}(p)
>0.
\]
Thus $[\cdot,\cdot]_{\mathcal S,W_N}^{(\zeta)}$ is an inner product on
$\Pi_{N-1}$.

Conversely, assume that $\mathcal G_{\mathcal S}$ is not connected. By
Theorem~\ref{thm:endpointTreeCriterion}, the matrix $A_{\mathcal S}$ is
singular. Hence, by the factorization~\eqref{eq:intervalGraphFactorization},
the matrix $M_{\mathcal S}^{(\zeta)}$ is singular. Therefore, there exists a
nonzero vector
\[
\boldsymbol a=\left(a_0,\ldots,a_{N-1}\right)^{\top}
\]
such that
\[
M_{\mathcal S}^{(\zeta)}\boldsymbol a=\boldsymbol 0.
\]
Let
\[
p=\sum_{k=0}^{N-1}a_kP_k.
\]
Since $P_0,\ldots,P_{N-1}$ is a basis and $\boldsymbol a\neq\boldsymbol 0$,
the polynomial $p$ is nonzero. On the other hand,
\[
\boldsymbol m_{\mathcal S}^{(\zeta)}(p)
=
M_{\mathcal S}^{(\zeta)}\boldsymbol a
=
\boldsymbol 0.
\]
Therefore
\[
[p,p]_{\mathcal S,W_N}^{(\zeta)}=0
\]
with $p\neq0$. Hence the bilinear form is not positive definite, and so it
is not an inner product. Consequently the equivalence is proven.
\end{proof}

Assume now that $\mathcal G_{\mathcal S}$ is connected. Then, by Proposition~\ref{prop:intervalGraphInnerProduct},
$[\cdot,\cdot]_{\mathcal S,W_N}^{(\zeta)}$ is an inner product on
$\Pi_{N-1}$. Applying the Gram--Schmidt orthogonalization to
\[
1,x,\ldots,x^{N-1}
\]
and normalizing the leading coefficient to one, we obtain a unique family
of monic polynomials
\[
\Phi_0,\Phi_1,\ldots,\Phi_{N-1},
\quad
\deg(\Phi_k)=k,
\quad
k=0,\ldots,N-1,
\]
such that
\begin{equation}
\label{eq:intervalGraphOrthogonality}
[\Phi_j,\Phi_k]_{\mathcal S,W_N}^{(\zeta)}
=
h_k\delta_{jk},
\quad
h_k>0,
\quad
j,k=0,\ldots,N-1.
\end{equation}
Let $M_{\mathcal S}^{\Phi,\zeta}\in\mathbb R^{N\times N}$ be the
corresponding histopolation matrix, namely
\[
\left[M_{\mathcal S}^{\Phi,\zeta}\right]_{i,k+1}
=
\int_{s_i}\Phi_k(x)\omega_{\alpha,\beta}^{(\zeta)}(x)dx,
\quad
i=1,\ldots,N,
\quad
k=0,\ldots,N-1.
\]
Then~\eqref{eq:intervalGraphOrthogonality} is equivalent to
\begin{equation*}
\left(M_{\mathcal S}^{\Phi,\zeta}\right)^{\top}
W_N
M_{\mathcal S}^{\Phi,\zeta}
=
\operatorname{diag}\left(h_0,\ldots,h_{N-1}\right),
\quad
h_k>0.
\end{equation*}
Thus the connectedness of $\mathcal G_{\mathcal S}$ yields a polynomial
basis for which the weighted Gram matrix is diagonal and positive definite.

\section{Generalized Jacobi weights and moment reduction}
\label{subsec:generalShiftedJacobiDifferenceQuotient}

The aim of this section is to relate histopolation for generalized Jacobi
weights to histopolation for the classical Jacobi weight. Furthermore, we use this
relation for constructing an exactly diagonal configuration for the Chebyshev
weight of the fourth kind.

Let $\alpha,\beta>-1$, and consider the Jacobi weight
\[
\omega_{\alpha,\beta}(x)
=
(1-x)^\alpha(1+x)^\beta.
\]
Let $p,\gamma_+,\gamma_-\in\mathbb N_0$ and set
\[
d=2p+\gamma_++\gamma_-.
\]
We define
\begin{equation}\label{eq:chi}
\chi_{p,\gamma_+,\gamma_-}(x)
=
x^{2p}(1-x)^{\gamma_+}(1+x)^{\gamma_-}
\end{equation}
and
\begin{equation}
\label{eq:generalizedJacobiWeight}
\omega_{\alpha,\beta}^{(p,\gamma_+,\gamma_-)}(x)
=
\chi_{p,\gamma_+,\gamma_-}(x)
\omega_{\alpha,\beta}(x).
\end{equation}
Since $x^{2p}=|x|^{2p}$, we have
\[
\omega_{\alpha,\beta}^{(p,\gamma_+,\gamma_-)}(x)
=
|x|^{2p}
\omega_{\alpha+\gamma_+,\beta+\gamma_-}(x).
\]
Thus, the weight in~\eqref{eq:generalizedJacobiWeight} belongs to
$L^1(-1,1)$, is nonnegative on $(-1,1)$, and is positive almost
everywhere in $(-1,1)$. Weights of this form are referred to as generalized Jacobi
weights~\cite{Locher:1989:FGJ,Gautschi:2004:OPC}. In the symmetric case
\[
\alpha=\beta=\lambda-\frac{1}{2},
\quad
\lambda>-\frac{1}{2},
\quad
\gamma_+=\gamma_-=0,
\]
the weight~\eqref{eq:generalizedJacobiWeight} reduces to
\[
|x|^{2p}(1-x^2)^{\lambda-1/2},
\]
which is the generalized Gegenbauer weight; see~\cite{Xu:2002:IFG}.

Let $P_n^{(\alpha,\beta)}$ denote the Jacobi polynomial of degree $n$
normalized by
\[
P_n^{(\alpha,\beta)}(1)
=
\frac{(\alpha+1)_n}{n!},
\]
and let $\kappa_n^{(\alpha,\beta)}\neq 0$ denote its leading coefficient. The leading coefficient of
$\chi_{p,\gamma_+,\gamma_-}$ is
\[
\sigma_{\gamma_+}
=
(-1)^{\gamma_+}.
\]
For $k\geq0$, we consider a polynomial of the form
\begin{equation}
\label{eq:generalizedJacobiNumerator}
\mathcal N_k^{(p,\gamma_+,\gamma_-)}(x)
=
\sum_{j=0}^{d}
a_{k,j}^{(p,\gamma_+,\gamma_-)}
P_{k+j}^{(\alpha,\beta)}(x),
\end{equation}
where
\begin{equation}
\label{eq:highestGeneralizedJacobiCoefficient}
a_{k,d}^{(p,\gamma_+,\gamma_-)}
=
\frac{
\sigma_{\gamma_+}
}{
\kappa_{k+d}^{(\alpha,\beta)}
}.
\end{equation}
If $d\geq1$, we require the remaining coefficients
\[
a_{k,0}^{(p,\gamma_+,\gamma_-)},\ldots,a_{k,d-1}^{(p,\gamma_+,\gamma_-)}
\]
to satisfy
\begin{equation}
\label{eq:generalizedJacobiVanishingConditions}
\begin{aligned}
\left(
\mathcal N_k^{(p,\gamma_+,\gamma_-)}
\right)^{(r)}(0)
&=
0,
&&
r=0,\ldots,2p-1,
\\
\left(
\mathcal N_k^{(p,\gamma_+,\gamma_-)}
\right)^{(r)}(1)
&=
0,
&&
r=0,\ldots,\gamma_+-1,
\\
\left(
\mathcal N_k^{(p,\gamma_+,\gamma_-)}
\right)^{(r)}(-1)
&=
0,
&&
r=0,\ldots,\gamma_--1.
\end{aligned}
\end{equation}
These are exactly $d=2p+\gamma_++\gamma_-$ conditions for the $d$
remaining coefficients. Whenever one of the integers $p$, $\gamma_+$,
and $\gamma_-$ is zero, the corresponding line of conditions is
understood to be absent.

The next lemma shows that these conditions uniquely determine the
coefficients and define a polynomial family which is orthogonal with
respect to the modified weight.

\begin{lemma}
\label{lemmaGeneralizedJacobiPolynomialBasis}
For any $k\geq0$, if $d\geq1$, the coefficients
\[
a_{k,0}^{(p,\gamma_+,\gamma_-)},\ldots,a_{k,d-1}^{(p,\gamma_+,\gamma_-)}
\]
are uniquely determined by
\eqref{eq:generalizedJacobiVanishingConditions}. In all cases, the
quotient
\begin{equation}
\label{eq:generalizedJacobiPolynomial}
\mathcal Q_k^{(p,\gamma_+,\gamma_-)}(x)
=
\frac{
\mathcal N_k^{(p,\gamma_+,\gamma_-)}(x)
}{
\chi_{p,\gamma_+,\gamma_-}(x)
}
\end{equation}
is a monic polynomial of degree $k$. If $k\geq1$, then
$\mathcal Q_k^{(p,\gamma_+,\gamma_-)}$ is orthogonal to $\Pi_{k-1}$ with
respect to
$\omega_{\alpha,\beta}^{(p,\gamma_+,\gamma_-)}$.
\end{lemma}
\begin{proof}
If $d=0$, then necessarily
\[
p=\gamma_+=\gamma_-=0.
\]
By~\eqref{eq:chi}, we have
\[
\chi_{0,0,0}(x)=1,
\]
and hence
\[
\mathcal Q_k^{(0,0,0)}(x)
=
\mathcal N_k^{(0,0,0)}(x)
=
a_{k,0}^{(0,0,0)}P_k^{(\alpha,\beta)}(x)
=
\frac{P_k^{(\alpha,\beta)}(x)}
{\kappa_k^{(\alpha,\beta)}}.
\]
Thus, $\mathcal Q_k^{(0,0,0)}$ is monic of degree $k$ and, for
$k\geq1$, it is orthogonal to $\Pi_{k-1}$ with respect to
$\omega_{\alpha,\beta}$. This proves the result when $d=0$.

Assume now that $d\geq1$. The conditions
in~\eqref{eq:generalizedJacobiVanishingConditions} form a linear system
with $d$ equations and $d$ unknowns. We prove that the corresponding
homogeneous system has only the trivial solution. Let
\begin{equation}\label{eq:R}
    \mathcal R(x)
=
\sum_{j=0}^{d-1}
b_jP_{k+j}^{(\alpha,\beta)}(x)
\end{equation}
satisfy
\begin{equation*}
\begin{cases}
\mathcal R^{(r)}(0)=0, & \quad r=0,\ldots,2p-1,\\
\mathcal R^{(r)}(1)=0, & \quad r=0,\ldots,\gamma_+-1,\\
\mathcal R^{(r)}(-1)=0, & \quad r=0,\ldots,\gamma_--1.
\end{cases}
\end{equation*}
Whenever one of the integers $p$, $\gamma_+$, and $\gamma_-$ is zero,
the corresponding conditions are understood to be absent. Then $\mathcal R$ has a zero of
multiplicity $2p$ at $0$, a zero of multiplicity $\gamma_+$ at $1$, and
a zero of multiplicity $\gamma_-$ at $-1$. Therefore,
$\chi_{p,\gamma_+,\gamma_-}$ divides $\mathcal R$, and hence
\begin{equation}\label{eq:RfactS}
    \mathcal R(x)
=
\chi_{p,\gamma_+,\gamma_-}(x)\mathcal S(x)
\end{equation}
for some polynomial $\mathcal S$. If $k=0$, by~\eqref{eq:R}, we have
\[
\deg\left(\mathcal R\right)\leq d-1.
\]
Since $\mathcal R$ is divisible by a polynomial of degree $d$, it follows
that $\mathcal R=0$. Assume now that $k\geq1$. Since
\[
\deg\left(\mathcal R\right)\leq k+d-1,
\]
we have
\[
\deg\left(\mathcal S\right)\leq k-1.
\]
Moreover,
\[
\mathcal R
\in
\operatorname{span}
\left\{
P_k^{(\alpha,\beta)},
\ldots,
P_{k+d-1}^{(\alpha,\beta)}
\right\}.
\]
Thus, by Jacobi orthogonality, we get
\begin{equation}\label{eq:ort1}
    \int_{-1}^{1}
\mathcal R(x)\mathcal S(x)
\omega_{\alpha,\beta}(x)dx
=
0.
\end{equation}
On the other hand, using~\eqref{eq:RfactS}, we have
\begin{eqnarray}\notag
\int_{-1}^{1}
\mathcal R(x)\mathcal S(x)
\omega_{\alpha,\beta}(x)dx
&=&
\int_{-1}^{1}
\left[\mathcal S(x)\right]^2
\chi_{p,\gamma_+,\gamma_-}(x)
\omega_{\alpha,\beta}(x)dx
\\
&=&
\int_{-1}^{1}
\left[\mathcal S(x)\right]^2
\omega_{\alpha,\beta}^{(p,\gamma_+,\gamma_-)}(x)dx. \label{eq:ort2}
\end{eqnarray}
Then, combining~\eqref{eq:ort1} and~\eqref{eq:ort2}, we obtain
\[
\int_{-1}^{1}
\left[\mathcal S(x)\right]^2
\omega_{\alpha,\beta}^{(p,\gamma_+,\gamma_-)}(x)dx=0.
\]
Since the modified weight~\eqref{eq:generalizedJacobiWeight} is positive
almost everywhere in $(-1,1)$, we conclude that $\mathcal S=0$, and
consequently $\mathcal R=0$. Together with the case $k=0$, this shows that $\mathcal R=0$ for any
$k\geq0$. By the linear independence of the polynomials
\[
P_k^{(\alpha,\beta)},\ldots,
P_{k+d-1}^{(\alpha,\beta)},
\]
it follows that
\[
b_0=\cdots=b_{d-1}=0.
\]
Hence, the homogeneous system has only the trivial solution. Therefore,
the coefficient matrix is nonsingular, and the coefficients
$a_{k,0}^{(p,\gamma_+,\gamma_-)},\ldots,
a_{k,d-1}^{(p,\gamma_+,\gamma_-)}$ are uniquely determined.

The conditions
in~\eqref{eq:generalizedJacobiVanishingConditions} also show that
$\mathcal N_k^{(p,\gamma_+,\gamma_-)}$ is divisible by
$\chi_{p,\gamma_+,\gamma_-}$. Moreover, the term with $j=d$ is the only
term of degree $k+d$ in
$\mathcal N_k^{(p,\gamma_+,\gamma_-)}$. Hence, by
\eqref{eq:highestGeneralizedJacobiCoefficient}, the leading coefficient of
$\mathcal N_k^{(p,\gamma_+,\gamma_-)}$ is
\[
a_{k,d}^{(p,\gamma_+,\gamma_-)}
\kappa_{k+d}^{(\alpha,\beta)}
=
\sigma_{\gamma_+}.
\]
In particular,
\[
\deg\left(
\mathcal N_k^{(p,\gamma_+,\gamma_-)}
\right)
=
k+d.
\]
Since
\[
\deg\left(
\chi_{p,\gamma_+,\gamma_-}
\right)
=
d,
\]
it follows that
\[
\deg\left(
\mathcal Q_k^{(p,\gamma_+,\gamma_-)}
\right)
=
k.
\]
Furthermore, $\chi_{p,\gamma_+,\gamma_-}$ has leading coefficient
$\sigma_{\gamma_+}$. Since
\[
\mathcal N_k^{(p,\gamma_+,\gamma_-)}
=
\chi_{p,\gamma_+,\gamma_-}
\mathcal Q_k^{(p,\gamma_+,\gamma_-)},
\]
the leading coefficient of
$\mathcal Q_k^{(p,\gamma_+,\gamma_-)}$ is one. Therefore,
$\mathcal Q_k^{(p,\gamma_+,\gamma_-)}$ is monic.

It remains to prove orthogonality. Assume that $k\geq1$, and let
\[
q\in\Pi_{k-1}.
\]
Using~\eqref{eq:generalizedJacobiWeight} and~\eqref{eq:generalizedJacobiPolynomial}, we obtain
\begin{equation*}
\int_{-1}^{1}
\mathcal Q_k^{(p,\gamma_+,\gamma_-)}(x)
q(x)
\omega_{\alpha,\beta}^{(p,\gamma_+,\gamma_-)}(x)dx=\int_{-1}^{1}
\mathcal N_k^{(p,\gamma_+,\gamma_-)}(x)
q(x)
\omega_{\alpha,\beta}(x)dx.
\end{equation*}
The last integral is zero, since
\[
\mathcal N_k^{(p,\gamma_+,\gamma_-)}\in\operatorname{span}
\left\{
P_k^{(\alpha,\beta)},
\ldots,
P_{k+d}^{(\alpha,\beta)}
\right\},
\]
whereas
\[
\deg\left(q\right)\leq k-1.
\]
Therefore,
$\mathcal Q_k^{(p,\gamma_+,\gamma_-)}$ is orthogonal to $\Pi_{k-1}$ with
respect to the modified weight~\eqref{eq:generalizedJacobiWeight}.
\end{proof}

\begin{remark}
For any $N\geq1$, since
\[
\deg\left(
\mathcal Q_k^{(p,\gamma_+,\gamma_-)}
\right)
=
k,
\quad
k=0,\ldots,N-1,
\]
the polynomials
\[
\mathcal Q_0^{(p,\gamma_+,\gamma_-)},
\ldots,
\mathcal Q_{N-1}^{(p,\gamma_+,\gamma_-)}
\]
form a basis of $\Pi_{N-1}$.
\end{remark}

\begin{remark}
Since $\mathcal Q_0^{(p,\gamma_+,\gamma_-)}$ is a monic polynomial of
degree zero, we have
\[
\mathcal Q_0^{(p,\gamma_+,\gamma_-)}(x)=1.
\]
If $p=0$, the modified weight~\eqref{eq:generalizedJacobiWeight} is again
a classical Jacobi weight, since
\[
\omega_{\alpha,\beta}^{(0,\gamma_+,\gamma_-)}(x)
=
\omega_{\alpha+\gamma_+,\beta+\gamma_-}(x).
\]
Therefore, by the uniqueness of the monic orthogonal polynomial of each
degree, we have
\begin{equation*}
\mathcal Q_k^{(0,\gamma_+,\gamma_-)}(x)
=
\frac{
P_k^{(\alpha+\gamma_+,\beta+\gamma_-)}(x)
}{
\kappa_k^{(\alpha+\gamma_+,\beta+\gamma_-)}
},
\quad
k\geq0.
\end{equation*}
\end{remark}

\subsection{Reduction of generalized Jacobi cell moments}
\label{subsec:generalizedJacobiMomentReduction}

Let
\[
s_i=\left[a_i,b_i\right]\subset(-1,1),
\quad
i=1,\ldots,N,
\]
and define the Jacobi cell moments by
\begin{equation}
\label{eq:referenceJacobiCellMoments}
\mathcal J_i(n)
=
\int_{s_i}
P_n^{(\alpha,\beta)}(x)
\omega_{\alpha,\beta}(x)dx,
\quad
n\geq0.
\end{equation}
Let $\mathcal M_N^{(p,\gamma_+,\gamma_-)}
\in\mathbb R^{N\times N}$ be the generalized Jacobi moment matrix defined by
\begin{equation*}
\left[
\mathcal M_N^{(p,\gamma_+,\gamma_-)}
\right]_{i,k+1}
=
\int_{s_i}
\mathcal Q_k^{(p,\gamma_+,\gamma_-)}(x)
\omega_{\alpha,\beta}^{(p,\gamma_+,\gamma_-)}(x)dx,
\end{equation*}
for $i=1,\ldots,N$, $k=0,\ldots,N-1$.

The following proposition shows that the cell moments associated with the
generalized Jacobi weight can be expressed as finite linear combinations
of classical Jacobi cell moments.

\begin{proposition}
\label{prop:generalizedJacobiMomentReduction}
For any $i=1,\ldots,N$ and $k=0,\ldots,N-1$, we have
\begin{equation}
\label{eq:generalizedJacobiMomentReduction}
\left[
\mathcal M_N^{(p,\gamma_+,\gamma_-)}
\right]_{i,k+1}
=
\sum_{j=0}^{d}
a_{k,j}^{(p,\gamma_+,\gamma_-)}
\mathcal J_i(k+j),
\end{equation}
where the coefficients
\[
a_{k,j}^{(p,\gamma_+,\gamma_-)},
\quad
j=0,\ldots,d,
\]
are those appearing in~\eqref{eq:generalizedJacobiNumerator}.
\end{proposition}

\begin{proof}
Using~\eqref{eq:generalizedJacobiWeight} and
\eqref{eq:generalizedJacobiPolynomial}, we obtain
\begin{equation*}
\mathcal Q_k^{(p,\gamma_+,\gamma_-)}(x)
\omega_{\alpha,\beta}^{(p,\gamma_+,\gamma_-)}(x)
=
\mathcal N_k^{(p,\gamma_+,\gamma_-)}(x)
\omega_{\alpha,\beta}(x).
\end{equation*}
Using~\eqref{eq:generalizedJacobiNumerator} and integrating over
$s_i$, we obtain
\begin{eqnarray*}
\left[
\mathcal M_N^{(p,\gamma_+,\gamma_-)}
\right]_{i,k+1}
&=&
\int_{s_i}
\mathcal N_k^{(p,\gamma_+,\gamma_-)}(x)
\omega_{\alpha,\beta}(x)dx
\\
&=&
\sum_{j=0}^{d}
a_{k,j}^{(p,\gamma_+,\gamma_-)}
\int_{s_i}
P_{k+j}^{(\alpha,\beta)}(x)
\omega_{\alpha,\beta}(x)dx
\\
&=&
\sum_{j=0}^{d}
a_{k,j}^{(p,\gamma_+,\gamma_-)}
\mathcal J_i(k+j),
\end{eqnarray*}
which proves~\eqref{eq:generalizedJacobiMomentReduction}.
\end{proof}

For any positive degree, the Jacobi cell moments admit an
explicit representation in terms of endpoint values.

\begin{lemma}
\label{lemmaJacobiCellMomentPrimitive}
For any $i=1,\ldots,N$ and $n\geq1$, we have
\begin{equation}
\label{eq:JacobiCellMomentPrimitive}
\mathcal J_i(n)
=
\left[
-\frac{1}{2n}
(1-x)^{\alpha+1}(1+x)^{\beta+1}
P_{n-1}^{(\alpha+1,\beta+1)}(x)
\right]_{a_i}^{b_i}.
\end{equation}
\end{lemma}

\begin{proof}
The weighted differentiation identity for Jacobi polynomials (see~\cite{Guo:2009:GJP}) gives
\[
\frac{d}{dx}
\left[
(1-x)^{\alpha+1}(1+x)^{\beta+1}
P_{n-1}^{(\alpha+1,\beta+1)}(x)
\right]
=
-2nP_n^{(\alpha,\beta)}(x)\omega_{\alpha,\beta}(x).
\]
By integrating this identity over $\left[a_i,b_i\right]$ and by dividing by $-2n$, we infer \eqref{eq:JacobiCellMomentPrimitive}.
\end{proof}

\subsection{An alternative basis for shifted Jacobi weights}
\label{subsec:telescopicOneSidedJacobiShift}

We now prove a general correction result for Jacobi polynomials. With the
notation introduced above, assume that
\[
d=2p+\gamma_++\gamma_-\ge 1,
\]
and define the space
\begin{equation*}
X_{N,d}^{(\alpha,\beta)}
=
\operatorname{span}
\left\{
P_N^{(\alpha,\beta)},\ldots,
P_{N+d-1}^{(\alpha,\beta)}
\right\}.
\end{equation*}
The next lemma shows that, for each Jacobi polynomial of degree less than $N$, there is a unique polynomial in $X_{N,d}^{(\alpha,\beta)}$
such that subtracting the given polynomial from it yields a multiple
of~\eqref{eq:chi}. The resulting family of quotient polynomials is then shown to form a basis of $\Pi_{N-1}$.

\begin{lemma}
\label{lemmaHighDegreeCorrectionPrinciple}
For any $k=0,\ldots,N-1$, there exists a unique polynomial
\[
\widetilde C_k\in X_{N,d}^{(\alpha,\beta)}
\]
such that
\begin{eqnarray}
\left(
\widetilde C_k-P_k^{(\alpha,\beta)}
\right)^{(r)}(0)
&=&
0,
\quad
r=0,\ldots,2p-1,
\notag\\
\left(
\widetilde C_k-P_k^{(\alpha,\beta)}
\right)^{(r)}(1)
&=&
0,
\quad
r=0,\ldots,\gamma_+-1,
\label{eq:highDegreeCorrectionConditions}\\
\left(
\widetilde C_k-P_k^{(\alpha,\beta)}
\right)^{(r)}(-1)
&=&
0,
\quad
r=0,\ldots,\gamma_--1.
\notag
\end{eqnarray}
The quotient
\begin{equation}
\label{eq:highDegreeCorrectedQuotient}
\mathcal{\widetilde Q}_k(x)
=
\frac{
\widetilde C_k(x)-P_k^{(\alpha,\beta)}(x)
}{
\chi_{p,\gamma_+,\gamma_-}(x)
},
\quad
k=0,\ldots,N-1,
\end{equation}
is therefore a polynomial belonging to $\Pi_{N-1}$.
Moreover, the families
\[
\mathcal B_1
=
\left\{
\mathcal{\widetilde Q}_0,\ldots,
\mathcal{\widetilde Q}_{N-1}
\right\}, \quad
\mathcal B_2
=
\left\{
1,\mathcal{\widetilde Q}_1,\ldots,
\mathcal{\widetilde Q}_{N-1}
\right\}
\]
are bases of $\Pi_{N-1}$.
\end{lemma}
\begin{proof}
By definition, any polynomial in
$X_{N,d}^{(\alpha,\beta)}$ can be written uniquely as
\[
C(x)
=
\sum_{j=0}^{d-1}
c_jP_{N+j}^{(\alpha,\beta)}(x).
\]
Substituting this expression into
\eqref{eq:highDegreeCorrectionConditions} gives a linear system for the
$d$ coefficients
\[
c_0,\ldots,c_{d-1}.
\]
Since the system consists of $d$ equations in $d$ unknowns, it is sufficient
to prove that the corresponding homogeneous system has only the trivial
solution. Let
\[
C\in X_{N,d}^{(\alpha,\beta)}
\]
satisfy
\begin{equation*}
\begin{cases}
C^{(r)}(0)=0, & \quad r=0,\ldots,2p-1,\\
C^{(r)}(1)=0, & \quad r=0,\ldots,\gamma_+-1,\\
C^{(r)}(-1)=0, & \quad r=0,\ldots,\gamma_--1.
\end{cases}
\end{equation*}
The argument used in the proof of
Lemma~\ref{lemmaGeneralizedJacobiPolynomialBasis}, applied with $k=N$,
shows that $C=0$. Therefore, the homogeneous system has only the trivial
solution, and the polynomial $\widetilde C_k$ exists and is uniquely
determined. The conditions~\eqref{eq:highDegreeCorrectionConditions} show that $\chi_{p,\gamma_+,\gamma_-}$ divides
\[
\widetilde C_k-P_k^{(\alpha,\beta)}.
\]
Hence, the quotient in~\eqref{eq:highDegreeCorrectedQuotient} is a
 polynomial. Moreover,
\[
\deg\left(
\widetilde C_k-P_k^{(\alpha,\beta)}
\right)
\leq
N+d-1,
\]
and therefore
\[
\deg\left(\mathcal{\widetilde Q}_k\right)
\leq
N-1.
\]
Thus,
\[
\mathcal{\widetilde Q}_k\in\Pi_{N-1},
\quad
k=0,\ldots,N-1.
\]
We now prove that
\[
\mathcal{\widetilde Q}_0,\ldots,\mathcal{\widetilde Q}_{N-1}
\]
are linearly independent. Assume that
\[
\sum_{k=0}^{N-1}a_k\mathcal{\widetilde Q}_k(x)=0.
\]
Multiplying by $\chi_{p,\gamma_+,\gamma_-}(x)$ and using
\eqref{eq:highDegreeCorrectedQuotient}, we obtain
\begin{equation}
\label{eq:highDegreeQuotientLinearRelation}
\sum_{k=0}^{N-1}a_k\widetilde C_k(x)
=
\sum_{k=0}^{N-1}
a_kP_k^{(\alpha,\beta)}(x).
\end{equation}
The left hand side belongs to $X_{N,d}^{(\alpha,\beta)}$, whereas the right hand side belongs to $\Pi_{N-1}$. Since
\[
X_{N,d}^{(\alpha,\beta)}
\cap
\Pi_{N-1}
=
\{0\},
\]
both sides of~\eqref{eq:highDegreeQuotientLinearRelation} vanish.
In particular,
\[
\sum_{k=0}^{N-1}
a_kP_k^{(\alpha,\beta)}(x)
=
0.
\]
The linear independence of the Jacobi polynomials then gives
\[
a_0=\cdots=a_{N-1}=0.
\]
Therefore,
$\mathcal{\widetilde Q}_0,\ldots,\mathcal{\widetilde Q}_{N-1}$
are linearly independent. Since they are $N$ polynomials in  $\Pi_{N-1}$, the family $\mathcal B_1$ is a basis
of $\Pi_{N-1}$.

Finally, we prove that $\mathcal B_2$ is a basis of $\Pi_{N-1}$. Assume that
\begin{equation}
\label{eq:constantHighDegreeQuotientRelation}
a_0
+
\sum_{k=1}^{N-1}
a_k\mathcal{\widetilde Q}_k(x)
=
0.
\end{equation}
Multiplying by $\chi_{p,\gamma_+,\gamma_-}(x)$ gives
\[
a_0\chi_{p,\gamma_+,\gamma_-}(x)
+
\sum_{k=1}^{N-1}
a_k\widetilde C_k(x)
-
\sum_{k=1}^{N-1}
a_kP_k^{(\alpha,\beta)}(x)
=
0.
\]
Integrating this identity over $[-1,1]$ with respect to
$\omega_{\alpha,\beta}$ and using Jacobi orthogonality, we obtain
\[
a_0
\int_{-1}^{1}
\chi_{p,\gamma_+,\gamma_-}(x)
\omega_{\alpha,\beta}(x)dx
=
0.
\]
Since
\[
\int_{-1}^{1}
\chi_{p,\gamma_+,\gamma_-}(x)
\omega_{\alpha,\beta}(x)dx
=
\int_{-1}^{1}
\omega_{\alpha,\beta}^{(p,\gamma_+,\gamma_-)}(x)dx
>
0,
\]
we obtain $a_0=0$. Hence,
equation~\eqref{eq:constantHighDegreeQuotientRelation} becomes
\[
\sum_{k=1}^{N-1}
a_k\mathcal{\widetilde Q}_k(x)
=
0.
\]
Since these polynomials belong to the linearly independent family
$\mathcal B_1$, it follows that
\[
a_1=\cdots=a_{N-1}=0.
\]
Hence, $\mathcal B_2$ also forms a basis of $\Pi_{N-1}$.
\end{proof}

We now apply Lemma~\ref{lemmaHighDegreeCorrectionPrinciple} when the modifying polynomial consists of a single endpoint factor. Let $N\geq2$ and fix $\zeta\in\{-1,1\}$. In the notation
of Lemma~\ref{lemmaHighDegreeCorrectionPrinciple}, we set
\[
\left(p,\gamma_+,\gamma_-\right)
=
\begin{cases}
(0,1,0), & \zeta=1,\\
(0,0,1), & \zeta=-1.
\end{cases}
\]
It follows that
\[
\chi_{p,\gamma_+,\gamma_-}(x)
=
1-\zeta x,
\quad
d=1,
\]
and hence
\[
X_{N,1}^{(\alpha,\beta)}
=
\operatorname{span}
\left\{
P_N^{(\alpha,\beta)}
\right\}.
\]
The corresponding shifted Jacobi weight is
\begin{equation}
\label{eq:oneStepShiftedJacobiWeight}
\omega_{\alpha,\beta}^{(\zeta)}(x)
=
(1-\zeta x)\omega_{\alpha,\beta}(x),
\end{equation}
that is
\[
\omega_{\alpha,\beta}^{(\zeta)}(x)
=
\begin{cases}
\omega_{\alpha+1,\beta}(x), & \zeta=1,\\
\omega_{\alpha,\beta+1}(x), & \zeta=-1.
\end{cases}
\]
In this case, the conditions~\eqref{eq:highDegreeCorrectionConditions} reduce to the single condition
\[
\widetilde C_k(\zeta)
=
P_k^{(\alpha,\beta)}(\zeta).
\]
The endpoint values of the Jacobi polynomials are
\[
P_k^{(\alpha,\beta)}(1)
=
\frac{(\alpha+1)_k}{k!},
\quad
P_k^{(\alpha,\beta)}(-1)
=
(-1)^k\frac{(\beta+1)_k}{k!}.
\]
Since $\alpha,\beta>-1$, we have
\[
P_k^{(\alpha,\beta)}(\zeta)\neq0,
\quad
k\geq0.
\]
We may therefore introduce the normalized Jacobi polynomials
\begin{equation}
\label{eq:endpointNormalizedJacobiPolynomials}
R_k^{(\zeta)}(x)
=
\frac{
P_k^{(\alpha,\beta)}(x)
}{
P_k^{(\alpha,\beta)}(\zeta)
},
\quad
k=0,\ldots,N.
\end{equation}
By construction,
\[
R_k^{(\zeta)}(\zeta)=1,
\quad
\deg\left(R_k^{(\zeta)}\right)=k.
\]
Since $X_{N,1}^{(\alpha,\beta)}$ is generated by
$P_N^{(\alpha,\beta)}$, the correction polynomial is given explicitly by
\[
\widetilde C_k(x)
=
\frac{
P_k^{(\alpha,\beta)}(\zeta)
}{
P_N^{(\alpha,\beta)}(\zeta)
}
P_N^{(\alpha,\beta)}(x)
=
P_k^{(\alpha,\beta)}(\zeta)
R_N^{(\zeta)}(x).
\]
Thus, for $k=0,\ldots,N-1$, the quotient polynomial defined in
Lemma~\ref{lemmaHighDegreeCorrectionPrinciple} satisfies
\begin{equation*}
\mathcal{\widetilde Q}_k(x)=
\frac{
\widetilde C_k(x)-P_k^{(\alpha,\beta)}(x)
}{
1-\zeta x
}
=
P_k^{(\alpha,\beta)}(\zeta)
\frac{
R_N^{(\zeta)}(x)-R_k^{(\zeta)}(x)
}{
1-\zeta x
}.
\end{equation*}
We now consider the basis obtained from $\mathcal B_2$ by rescaling its
nonconstant elements. Specifically, we set
\[
Q_0^{(\zeta)}(x)=1,
\]
and consider
\begin{equation}
\label{eq:shiftedJacobiDifferenceQuotients}
Q_k^{(\zeta)}(x)
=
\frac{
R_N^{(\zeta)}(x)-R_k^{(\zeta)}(x)
}{
1-\zeta x
}, \quad k=1,\ldots,N-1.
\end{equation}
The identity above gives
\[
Q_k^{(\zeta)}(x)
=
\frac{
\mathcal{\widetilde Q}_k(x)
}{
P_k^{(\alpha,\beta)}(\zeta)
},
\quad
k=1,\ldots,N-1.
\]
Since $P_k^{(\alpha,\beta)}(\zeta)\neq0$ for
$k=1,\ldots,N-1$, the family
\begin{equation*}
    \mathcal{B}=\left\{Q_0^{(\zeta)},\ldots,Q_{N-1}^{(\zeta)}\right\}
\end{equation*}
is obtained from the basis $\mathcal B_2$ by rescaling its nonconstant
elements. Therefore, it forms a basis of $\Pi_{N-1}$. Finally, since
\[
R_N^{(\zeta)}(\zeta)
=
R_k^{(\zeta)}(\zeta)
=
1,
\]
the numerator in~\eqref{eq:shiftedJacobiDifferenceQuotients} vanishes at
$x=\zeta$. Moreover, since $k<N$, it has degree $N$, and hence
\[
\deg\left(Q_k^{(\zeta)}\right)
=
N-1,
\quad
k=1,\ldots,N-1.
\]

Let $M_N^{(\zeta)}\in \mathbb R^{N\times N}$
be the moment matrix defined by
\begin{equation*}
\left[
M_N^{(\zeta)}
\right]_{i,k+1}
=
\int_{s_i}
Q_k^{(\zeta)}(x)
\omega_{\alpha,\beta}^{(\zeta)}(x)dx,
\quad
i=1,\ldots,N,
\quad
k=0,\ldots,N-1.
\end{equation*}
For $k=0$, using $Q_0^{(\zeta)}=1$, we have
\begin{equation}
\label{eq:shiftedJacobiFirstColumn}
\left[
M_N^{(\zeta)}
\right]_{i,1}
=
\int_{s_i}
(1-\zeta x)\omega_{\alpha,\beta}(x)dx,
\quad
i=1,\ldots,N.
\end{equation}
Using the identity
\[
P_1^{(\alpha,\beta)}(x)
=
\frac{1}{2}
\left[
(\alpha+\beta+2)x+\alpha-\beta
\right],
\]
we can write
\[
x
=
\frac{
2P_1^{(\alpha,\beta)}(x)-\alpha+\beta
}{
\alpha+\beta+2
}.
\]
Substituting this expression into~\eqref{eq:shiftedJacobiFirstColumn} and using~\eqref{eq:referenceJacobiCellMoments}, we obtain
\begin{equation*}
\left[
M_N^{(\zeta)}
\right]_{i,1}
=
\left(
1+
\zeta\frac{\alpha-\beta}{\alpha+\beta+2}
\right)
\mathcal J_i(0)
-
\frac{2\zeta}{\alpha+\beta+2}
\mathcal J_i(1).
\end{equation*}
For $k=1,\ldots,N-1$, using~\eqref{eq:oneStepShiftedJacobiWeight}, \eqref{eq:endpointNormalizedJacobiPolynomials} and~\eqref{eq:shiftedJacobiDifferenceQuotients}, we obtain
\begin{equation*}
\left[
M_N^{(\zeta)}
\right]_{i,k+1}
=
\int_{s_i}
\left[
R_N^{(\zeta)}(x)-R_k^{(\zeta)}(x)
\right]
\omega_{\alpha,\beta}(x)dx
=
\frac{
\mathcal J_i(N)
}{
P_N^{(\alpha,\beta)}(\zeta)
}
-
\frac{
\mathcal J_i(k)
}{
P_k^{(\alpha,\beta)}(\zeta)
}.
\end{equation*}
By Lemma~\ref{lemmaJacobiCellMomentPrimitive}, these entries can be
obtained explicitly from the endpoints of the cells, without computing
the integrals directly.

\subsection{A configuration for the fourth kind Chebyshev weight}
\label{subsec:fourthKindChebyshevConfiguration}

We now apply the preceding construction to the case
\[
\alpha=\beta=-\frac{1}{2},
\quad
\zeta=1.
\]
Throughout this subsection, we assume that $N\geq2$.
The original weight is the Chebyshev weight of the first kind, that is
\[
\omega_{-1/2,-1/2}(x)
=
\frac{1}{\sqrt{1-x^2}},
\]
whereas the shifted weight defined in~\eqref{eq:oneStepShiftedJacobiWeight} becomes
\begin{equation*}
\omega_{-1/2,-1/2}^{(1)}(x)
=
(1-x)\omega_{-1/2,-1/2}(x)
=
\left(
\frac{1-x}{1+x}
\right)^{1/2}
=
\omega_{1/2,-1/2}(x).
\end{equation*}
Thus, the shifted weight is the Chebyshev weight of the fourth kind.

Let $T_k$ denote the Chebyshev polynomial of the first kind of degree $k$.
Using the classical Jacobi normalization, we have
\[
P_k^{(-1/2,-1/2)}(x)
=
a_kT_k(x),
\quad
a_k
=
\frac{(1/2)_k}{k!}.
\]
Since
\[
P_k^{(-1/2,-1/2)}(1)=a_k,
\]
the normalized polynomials defined in~\eqref{eq:endpointNormalizedJacobiPolynomials} become
\[
R_k^{(1)}(x)=T_k(x), \quad k=0,\dots,N.
\]
Consequently, the basis introduced in~\eqref{eq:shiftedJacobiDifferenceQuotients} becomes
\begin{equation}
\label{eq:fourthKindPolynomialBasis}
Q_0^{(1)}(x)=1,
\quad
Q_k^{(1)}(x)
=
\frac{
T_N(x)-T_k(x)
}{
1-x
},
\quad
k=1,\ldots,N-1.
\end{equation}
For
\[
x=\cos\theta,
\quad
0<\theta<\pi,
\]
we have
\[
\omega_{1/2,-1/2}(\cos\theta)\sin\theta
=\left(
\frac{1-\cos\theta}{1+\cos\theta}
\right)^{1/2}\sin\theta=
1-\cos\theta.
\]
It follows from~\eqref{eq:fourthKindPolynomialBasis} that
\begin{equation}
\label{eq:fourthKindMomentIdentity}
Q_k^{(1)}(\cos\theta)
\omega_{1/2,-1/2}(\cos\theta)\sin\theta
=
\cos(N\theta)-\cos(k\theta), \quad k=1,\ldots,N-1.
\end{equation}
Let
\[
0<\rho<\frac{\pi}{2N}.
\]
We consider the nodes
\[
\tau_i
=
\frac{(2i-1)\pi}{2N},
\quad
i=1,\ldots,N,
\]
and the corresponding cells
\[
s_i
=
\left[
\cos\left(\tau_i+\rho\right),
\cos\left(\tau_i-\rho\right)
\right],
\quad
i=1,\ldots,N.
\]
Let $M_N\in\mathbb R^{N\times N}$ be the unnormalized moment matrix
defined by
\[
\left[M_N\right]_{i,k+1}
=
\int_{s_i}
Q_k^{(1)}(x)\omega_{1/2,-1/2}(x)dx, \quad i=1,\ldots,N,
\,
k=0,\ldots,N-1.
\]
We denote its columns by
\[
\boldsymbol c_k
=
\left(
\left[M_N\right]_{1,k+1},
\ldots,
\left[M_N\right]_{N,k+1}
\right)^{\top},
\quad
k=0,\ldots,N-1.
\]

In the following lemma, we express the columns of $M_N$ in terms of vectors obtained by sampling cosine functions at the points $\tau_1,\ldots,\tau_N$.

\begin{lemma}
\label{lemmaFourthKindColumnReduction}
Let
\[
\boldsymbol u_0
=
(1,\ldots,1)^{\top}, \quad
\boldsymbol u_k
=
\left(
\cos\left(k\tau_1\right),
\ldots,
\cos\left(k\tau_N\right)
\right)^{\top},
\quad
k=1,\ldots,N-1.
\]
Then
\begin{equation}
\label{eq:fourthKindConstantColumn}
\boldsymbol c_0
=
2\rho\boldsymbol u_0
-
2\sin(\rho)\boldsymbol u_1,
\end{equation}
whereas
\begin{equation}
\label{eq:fourthKindNonconstantColumns}
\boldsymbol c_k
=
-\frac{2\sin(k\rho)}{k}\boldsymbol u_k,
\quad
k=1,\ldots,N-1.
\end{equation}
\end{lemma}

\begin{proof}
For the first column, using the change of variables $x=\cos\theta$, we have
\begin{eqnarray*}
\left[\boldsymbol c_0\right]_i
&=&
\int_{s_i}
\omega_{1/2,-1/2}(x)dx=
\int_{\tau_i-\rho}^{\tau_i+\rho}
\omega_{1/2,-1/2}(\cos\theta)\sin\theta\,d\theta
\\
&=&
\int_{\tau_i-\rho}^{\tau_i+\rho}
(1-\cos\theta)d\theta
=
2\rho
-
2\sin(\rho)\cos\left(\tau_i\right),
\quad
i=1,\ldots,N.
\end{eqnarray*}
This proves~\eqref{eq:fourthKindConstantColumn}.

For $k=1,\ldots,N-1$, using
\eqref{eq:fourthKindMomentIdentity}, we obtain
\begin{eqnarray*}
    \left[\boldsymbol c_k\right]_i
&=& \int_{s_i}
Q_k^{(1)}(x)\omega_{1/2,-1/2}(x)dx= \int_{\tau_i-\rho}^{\tau_i+\rho}
Q_k^{(1)}(\cos \theta)\omega_{1/2,-1/2}(\cos \theta) \sin \theta d\theta\\ &=&
\int_{\tau_i-\rho}^{\tau_i+\rho}
\left[
\cos(N\theta)-\cos(k\theta)
\right]d\theta \\
&=&
\frac{2\sin(N\rho)}{N}\cos\left(N\tau_i\right)
-
\frac{2\sin(k\rho)}{k}\cos\left(k\tau_i\right).
\end{eqnarray*}
Since
\[
N\tau_i
=
\frac{(2i-1)\pi}{2},
\]
we have
\[
\cos\left(N\tau_i\right)=0.
\]
Therefore, we get
\[
\left[\boldsymbol c_k\right]_i
=
-\frac{2\sin(k\rho)}{k}\cos\left(k\tau_i\right),
\quad
i=1,\ldots,N,
\quad
k=1,\ldots,N-1,
\]
which proves~\eqref{eq:fourthKindNonconstantColumns}.
\end{proof}

Using~\eqref{cosineSumZero} and~\eqref{cosineDiscreteOrthogonality}, the vectors
\[
\boldsymbol u_0,\ldots,\boldsymbol u_{N-1}
\]
are mutually orthogonal. More precisely, for any $j,k=0,\ldots,N-1$, we have
\[
\boldsymbol u_j^{\top}\boldsymbol u_k
=
\begin{cases}
N, & j=k=0,\\[1mm]
\dfrac{N}{2}, & j=k\in\{1,\ldots,N-1\},\\[2mm]
0, & j\neq k.
\end{cases}
\]
To apply Theorem~\ref{thm:singleColumnCorrection}, we consider
\[
\Xi_N
=
\left\{
\tau_1,\ldots,\tau_N
\right\}
\]
and define
\[
\psi_k(\xi)
=
\cos(k\xi),
\quad
\xi\in\Xi_N,
\quad
k=0,\ldots,N-1.
\]
The corresponding sampling vectors are
\[
\boldsymbol u_k=\left(
\psi_k\left(\tau_1\right),\ldots,\psi_k\left(\tau_N\right)
\right)^{\top},
\quad
k=0,\ldots,N-1.
\]
Thus, in the notation of
\eqref{eq:generalWeightedOrthogonality}, we set
\[
W_N=I_N,
\]
\begin{equation}\label{newnu}
    \nu_0=N,
\quad
\nu_k=\frac{N}{2},
\quad
k=1,\ldots,N-1.
\end{equation}
Moreover, Lemma~\ref{lemmaFourthKindColumnReduction} gives
\[
\boldsymbol c_0
=
\eta_0\boldsymbol u_0+\eta_1\boldsymbol u_1,
\]
where
\begin{equation}\label{neweta}
    \eta_0=2\rho,
\quad
\eta_1=-2\sin(\rho),
\end{equation}
and
\[
\boldsymbol c_k
=
\widehat a_k\boldsymbol u_k,
\quad
k=1,\ldots,N-1,
\]
where
\begin{equation}\label{newak}
    \widehat a_k
=
-\frac{2\sin(k\rho)}{k},
\quad
k=1,\ldots,N-1.
\end{equation}
Therefore, the index appearing in
Theorem~\ref{thm:singleColumnCorrection} is
\[
m=1.
\]
This choice is admissible since $N\geq2$. Furthermore,
\[
\eta_0=2\rho>0.
\]
For $k=1,\ldots,N-1$, we also have
\[
0<k\rho
<
\frac{k\pi}{2N}
\leq
\frac{(N-1)\pi}{2N}
<
\frac{\pi}{2},
\]
and hence
\[
\widehat a_k\neq0,
\quad
k=1,\ldots,N-1.
\]
Thus, the column representations required in
Theorem~\ref{thm:singleColumnCorrection} hold with $m=1$.

Using~\eqref{eq:fourthKindConstantColumn},
\eqref{eq:fourthKindNonconstantColumns}, and the orthogonality of
$\boldsymbol u_0$ and $\boldsymbol u_1$, we obtain
\[
\boldsymbol c_0^{\top}\boldsymbol c_1
=
2N\sin^2(\rho)
\]
and
\[
\boldsymbol c_0^{\top}\boldsymbol c_0
=
2N
\left[
2\rho^2+\sin^2(\rho)
\right].
\]
Consequently, the correction parameter defined in
\eqref{eq:generalLambda} is
\begin{equation*}
\lambda_1
=
\frac{
\boldsymbol c_0^{\top}\boldsymbol c_1
}{
\boldsymbol c_0^{\top}\boldsymbol c_0
}
=
\frac{
\sin^2(\rho)
}{
2\rho^2+\sin^2(\rho)
}.
\end{equation*}
We now define
\[
\widehat Q_0^{(1)}(x)
=
Q_0^{(1)}(x)
=
1,
\]
\[
\widehat Q_1^{(1)}(x)
=
Q_1^{(1)}(x)-\lambda_1,
\]
and
\[
\widehat Q_k^{(1)}(x)
=
Q_k^{(1)}(x),
\quad
k=2,\ldots,N-1.
\]
Since $\widehat Q_1^{(1)}$ is obtained by subtracting a constant
multiple of $Q_0^{(1)}$ from $Q_1^{(1)}$, the family
\[
\left\{
\widehat Q_0^{(1)},\ldots,
\widehat Q_{N-1}^{(1)}
\right\}
\]
is again a basis of $\Pi_{N-1}$. Let now $\hat M_N\in\mathbb R^{N\times N}$ be the corresponding moment
matrix, namely
\[
\left[\hat M_N\right]_{i,k+1}
=
\int_{s_i}
\widehat Q_k^{(1)}(x)
\omega_{1/2,-1/2}(x)dx, \quad i=1,\ldots,N,
\,
k=0,\ldots,N-1.
\]
We can now apply
Theorem~\ref{thm:singleColumnCorrection} to the matrix $\hat M_N$.

\begin{theorem}
\label{thmFourthKindDiagonalGram}
Let $N\geq2$. For any
\[
0<\rho<\frac{\pi}{2N},
\]
the matrix $\hat M_N$ satisfies
\[
\hat M_N^{\top}\hat M_N
=
\operatorname{diag}
\left(
d_0,\ldots,d_{N-1}
\right),
\]
where
\[
d_0
=
2N
\left[
2\rho^2+\sin^2(\rho)
\right],
\quad
d_1
=
\frac{
4N\rho^2\sin^2(\rho)
}{
2\rho^2+\sin^2(\rho)
},
\]
and
\[
d_k
=
\frac{N}{2}
\left[
\frac{2\sin(k\rho)}{k}
\right]^2,
\quad
k=2,\ldots,N-1.
\]
In particular, $\hat M_N$ is nonsingular.
\end{theorem}

\begin{proof}
The result follows from
Theorem~\ref{thm:singleColumnCorrection}, applied with the parameters
identified above. In particular, the diagonal entries are
\[
d_0
=
\eta_0^2\nu_0+\eta_1^2\nu_1,
\quad
d_1
=
\frac{
\widehat a_1^2\nu_0\nu_1\eta_0^2
}{
\eta_0^2\nu_0+\eta_1^2\nu_1
},
\]
and
\[
d_k
=
\widehat a_k^2\nu_k,
\quad
k=2,\ldots,N-1.
\]
The stated formulas then follow by substituting the definitions of
$\eta_0$ and $\eta_1$ from~\eqref{neweta},
$\widehat a_k$ from~\eqref{newak}, and
$\nu_k$ from~\eqref{newnu}. Since
\[
0<k\rho<\frac{\pi}{2},
\quad
k=1,\ldots,N-1,
\]
all the diagonal entries are strictly positive. Therefore,
$\hat M_N^{\top}\hat M_N$ is positive definite, and hence
$\hat M_N$ is nonsingular.
\end{proof}

\section{Conclusions and Future Work}\label{sec6}

In the present work, we studied univariate polynomial weighted histopolation on
families of intervals in $[-1,1]$, with particular attention to unisolvence
and to configurations for which the associated Gram matrix is exactly
diagonal. For interval families whose endpoints belong to a fixed grid, we
proved that unisolvence is equivalent to the connectedness of the associated
endpoint graph. When this condition holds, the graph has $N+1$ vertices and
$N$ edges and is therefore a tree. The underlying structure gives an explicit
description of the inverse of the interval matrix in terms of the unique paths
joining consecutive grid points. In particular, the infinity norm of the
inverse is determined by the lengths of these paths. The same representation also provides a characterization of the spectral condition number of the interval matrix and a lower bound in terms of the maximum interval width and
the maximum path length in the endpoint tree. Furthermore, we considered cells of constant angular length. For the Chebyshev weight
of the first kind, discrete cosine orthogonality gives an exactly diagonal
Gram matrix for the normalized histopolation matrix. This allowed us to derive
explicit formulas for its diagonal entries, singular values, and spectral
condition number. We also proved that, whenever the number of cells is at
least three, requiring this diagonal structure for every admissible angular
half-length characterizes the Jacobi parameters
$\alpha=\beta=-\frac{1}{2}$. A general diagonalization criterion based on discrete weighted orthogonality was established, together with a correction result for a single nonconstant column. These results recover the first kind Chebyshev construction and provide
further configurations, including the constant weight case based on discrete
sine orthogonality. For interval families with a connected endpoint graph, we
also constructed a polynomial basis for which the corresponding weighted Gram
matrix is diagonal. Finally, for generalized Jacobi weights, we showed that the corresponding cell moments can be reduced to finite linear combinations of classical Jacobi cell moments and that, for positive degrees, the latter admit explicit representations in terms of endpoint values.  We also introduced an alternative basis for shifted Jacobi weights. For the Chebyshev weight of the fourth kind, taking into account the correction result, the considered basis leads to an exactly
diagonal Gram matrix.

Several questions remain open. It would be interesting to derive explicit
formulas or recurrence relations for the orthogonal basis associated with the
endpoint graph and to determine more precisely how the topology of the corresponding tree affects the diagonal entries and the spectral condition number of the moment
matrix. Finally, it would be interesting to use the moment reduction formulas for
generalized Jacobi weights to obtain explicit orthogonal bases for interval
families with a connected endpoint graph.

\section*{Declarations}

\textbf{Corresponding author}\\
Federico Nudo, email federico.nudo@unical.it\\

\noindent
\textbf{Conflict of Interest}\\
The authors declare that they have no conflict of interest.\\

\noindent
\textbf{Funding statement}\\
This research was supported by the GNCS-INdAM 2026 projects
\emph{``Metodi polinomiali e kernel per l'approssimazione da dati discreti e integrali con software OS''} and \emph{``Metodi strutturati per il signal processing avanzato''} (CUP\_E53C25002010001).
The work of F. Nudo was funded by the European Union -- NextGenerationEU under the Italian National Recovery and Resilience Plan (PNRR), Mission 4, Component 2, Investment 1.2
\lq\lq Finanziamento di progetti presentati da giovani ricercatori\rq\rq,
pursuant to MUR Decree No.~47/2025.\\

\noindent
\textbf{Author Contributions}\\
Allal Guessab, Federico Nudo and Stefano Serra-Capizzano contributed equally to the conception, development, and writing of this manuscript.
All authors have read and approved the final version of the paper. For this reason, the order of authorship is alphabetical.\\

\noindent
\textbf{Acknowledgements}\\
This research was carried out as part of RITA \textquotedblleft Research ITalian network on Approximation'' and as part of the UMI group \enquote{Teoria dell'Approssimazione e Applicazioni}.\\

\noindent
\textbf{Data Availability}\\
No data were used in this study.

\bibliographystyle{spmpsci}
\bibliography{bibliography}

\end{document}